\documentclass{siamart0516}

\usepackage{graphicx}
\usepackage{epstopdf}
\usepackage{amsmath}
\usepackage{amssymb}
\usepackage{multirow,booktabs}
\usepackage{xcolor}

\newcommand{\bbR}{\ensuremath{\mathbb{R}}}
\newcommand{\bbC}{\ensuremath{\mathbb{C}}}

\newcommand{\bigO}[1]{\ensuremath{\mathcal{O}\left( #1 \right)}}

\newcommand{\dd}{\ensuremath{\mathrm{d}}}

\newcommand{\realPart}[1]{\ensuremath{\mathrm{Re}\left[ #1 \right]}}

\newcommand{\mfd}{\ensuremath{\mathcal{M}}}
\newcommand{\eqMfd}{\ensuremath{\mathcal{E}}}

\newcommand{\delp}{\ensuremath{\Delta \varphi}}
\newcommand{\pbar}{\ensuremath{\bar{\varphi}}}

\newcommand{\delm}{\ensuremath{\Delta m}}
\newcommand{\mbar}{\ensuremath{\bar{m}}}

\newcommand{\delv}{\ensuremath{\Delta v}}
\newcommand{\vbar}{\ensuremath{\bar{v}}}

\DeclareMathOperator{\tr}{tr}

\newcommand{\map}[3]{#1: #2 \rightarrow #3}

\newcommand{\domain}{\ensuremath{\mathcal{D}}}
\newcommand{\oneto}[1]{\ensuremath{ \{1, \ldots, #1 \} }}

\newcommand{\tran}{T}

\def \etal {\emph{et al.}}

\newcommand{\beq}{\begin{equation}}
\newcommand{\eeq}{\end{equation}}
\newcommand{\bal}{\begin{align}}
\newcommand{\eal}{\end{align}}

\newsiamthm{conjecture}{Conjecture}
\newtheorem{remark}[theorem]{Remark}
\newtheorem{claim}[theorem]{Claim}

\usepackage{soul}

\setstcolor{red}

\title{\LARGE \bf
A dynamical system for prioritizing and coordinating motivations}

\author{Paul Reverdy and Daniel E. Koditschek\thanks{Department of Electrical and Systems Engineering, University of Pennsylvania, Philadelphia, PA (\email{preverdy@seas.upenn.edu}, \email{kod@seas.upenn.edu}). P. Reverdy's present address is Department of Aerospace and Mechanical Engineering, University of Arizona, Tucson, AZ (\email{preverdy@email.arizona.edu}).}}

\begin{document}
\maketitle
\begin{abstract}
We develop a dynamical systems approach to prioritizing and selecting multiple recurring tasks with the aim of conferring a degree of deliberative goal selection to a mobile robot confronted with competing objectives. We take navigation as our prototypical task, and use reactive (i.e., vector field) planners derived from navigation functions to encode control policies that achieve each individual task. We associate a scalar ``value'' with each task representing its current urgency and let that quantity evolve in time as the robot evaluates the importance of its assigned task relative to competing tasks. The robot's motion control input is generated as a convex combination of the individual task vector fields. Their weights, in turn, evolve dynamically according to a decision model adapted from the literature on bioinspired swarm decision making, driven by the values. In this paper we study a simple case with two recurring, competing navigation tasks and derive conditions under which it can be guaranteed that the robot will repeatedly serve each in turn. Specifically, we provide conditions sufficient for the emergence of a stable limit cycle along which the robot repeatedly and alternately navigates to the two goal locations. Numerical study suggests that the basin of attraction is quite large so that significant perturbations are recovered with a reliable return to the desired task coordination pattern.
\end{abstract}



\section{Introduction}\label{sec:introduction}
This paper develops a dynamical systems framework for planning and executing recurrent coverage tasks such as might be assigned a robot night watchman, sentry, or other patrol agent. Such problems have historically occasioned the introduction of temporal logic \cite{JL-etal:13} or probabilistic \cite{RL-etal:17} specification with subsequent hybrid implementation \cite{HKG-GEF-GJP:09}. We seek to explore a more ``embodied'' approach, i.e., one that is capable of direct integration with a robot's real-time sensorimotor models and controllers yet still offering a degree of deliberative judgment. A prototypical example of a real-time deliberative system is a foraging animal that achieves its basic needs for food and shelter by periodically revisiting different locations in its environment at different times apparently governed by some internal sense of relative urgency or satiety. In the vocabulary of psychology, the animal can be said to have drives which motivate it to perform actions that reduce those drives \cite{RSW:18, CLH:43, HLP-JMG:12}. Inspired by the flexibility and robustness of such natural ``reactively deliberate'' systems, we seek a simple model of their drive-based decision-making mechanisms that might be robustly embodied within the dynamical sensorimotor layers of autonomous physical systems --- a motivational dynamics for robots.

Dynamical systems approaches have been successful in understanding mechanisms for decision making in biological systems such as human choice behavior in two-alternative forced choice tasks \cite{RB-etal:06}, migration behavior in animal groups \cite{NEL-etal:12}, and nest site selection behavior \cite{TDS-etal:12} in honeybee swarms. This literature has further motivated the exploration of such bio-inspired models for application to engineered systems \cite{RG-etal:18}. Often, these decision mechanisms are \emph{value based} in the sense that the organism can be interpreted as associating a numerical value with each available alternative and selecting the alternative with the highest value. Decision making in biological systems tends to be \emph{embodied} in the sense that animals implement their decisions by moving their bodies in some way \cite{NFL-GP:15}. In the standard two-alternative forced choice task, an animal registers a decision by pushing a button or by looking at a particular point on a screen. In the context of migration or nest site selection, the animal moves its entire body to a new location. We take navigation, interpreted broadly as the task of steering a system to a desired goal state while avoiding obstacles, as the prototypical task for a mobile robot.

Vector field methods provide a natural way to encode the sensorimotor activity required to perform navigation tasks in dynamical systems language. When the vector field arises as the gradient of a well-chosen function, such as a navigation function \cite{ER-DEK:92}, the system dynamics readily admit performance guarantees, such as proofs of convergence to the desired state while avoiding obstructions along the way. Furthermore, such vector field methods naturally map to control inputs for mechanical systems described by Lagrangian dynamics \cite{DEK:89} and can be composed via linear combination or more intricate sequential \cite{RRB-AAR-DEK:99} and parallel \cite{AD-DEK:15} operations. 

These features motivate the consideration of vector field methods as an interesting alternative to the logical approaches to deliberation mentioned above, e.g., \cite{HKG-GEF-GJP:09}, whose hybrid (event-based) transitions require separately derived logical representation of the underlying dynamics. Instead, we seek an intrinsically dynamical systems approach to the composition and prioritization of potentially competing tasks that interprets the coefficients of their representative fields' linear combinations as a kind of motivational state to be continuously adjusted in real time in a way that is flexible and robust to perturbations. In this initial work we seek to merely encode an activity composed of cyclicly-repeating base tasks, targeting, e.g., the recurrent patrol missions mentioned at the outset. The natural dynamical systems object by which to encode such an activity is a limit cycle. In subsequent work \cite{PR-etal:18} we are studying empirical implementation of these ideas on physical robots and seek to enhance our limit cycle framework to allow the system to respond to external stimuli while keeping the limit cycle as the base behavior.

The main result of this paper is captured in Figure \ref{fig:bifurcationValue} which summarizes a numerical study illustrating two central analytical insights stated as Theorem \ref{thm:hopfResult} and Theorem \ref{thm:mainResult}. The motivational feedback path has a gain parametrized by $\epsilon_v > 0$ and a time scale parametrized by $\epsilon_{\lambda} > 0$. Numerical studies summarized by the four subsequent plots referenced by the numbered points of the figure indicate the presence of a stable limit cycle for a wide range of these parameter values. Analysis reveals that $\epsilon_v$ plays the role of an $\epsilon_\lambda$-dependent bifurcation parameter. Specifically, in the fast timescale limit $\epsilon_\lambda \to 0$, Theorem \ref{thm:hopfResult} establishes the existence of a Hopf bifurcation at a critical value of the feedback gain parameter $\epsilon^*_v(0)$. Further numerical study confirms the value of that formally-determined parameter, and suggests that the Hopf bifurcation persists along a curve of critical values, $\epsilon^*_v(\epsilon_\lambda)$ for positive $\epsilon_\lambda$.

\begin{figure}[htbp]
\begin{center}
\includegraphics[width=3.5in]{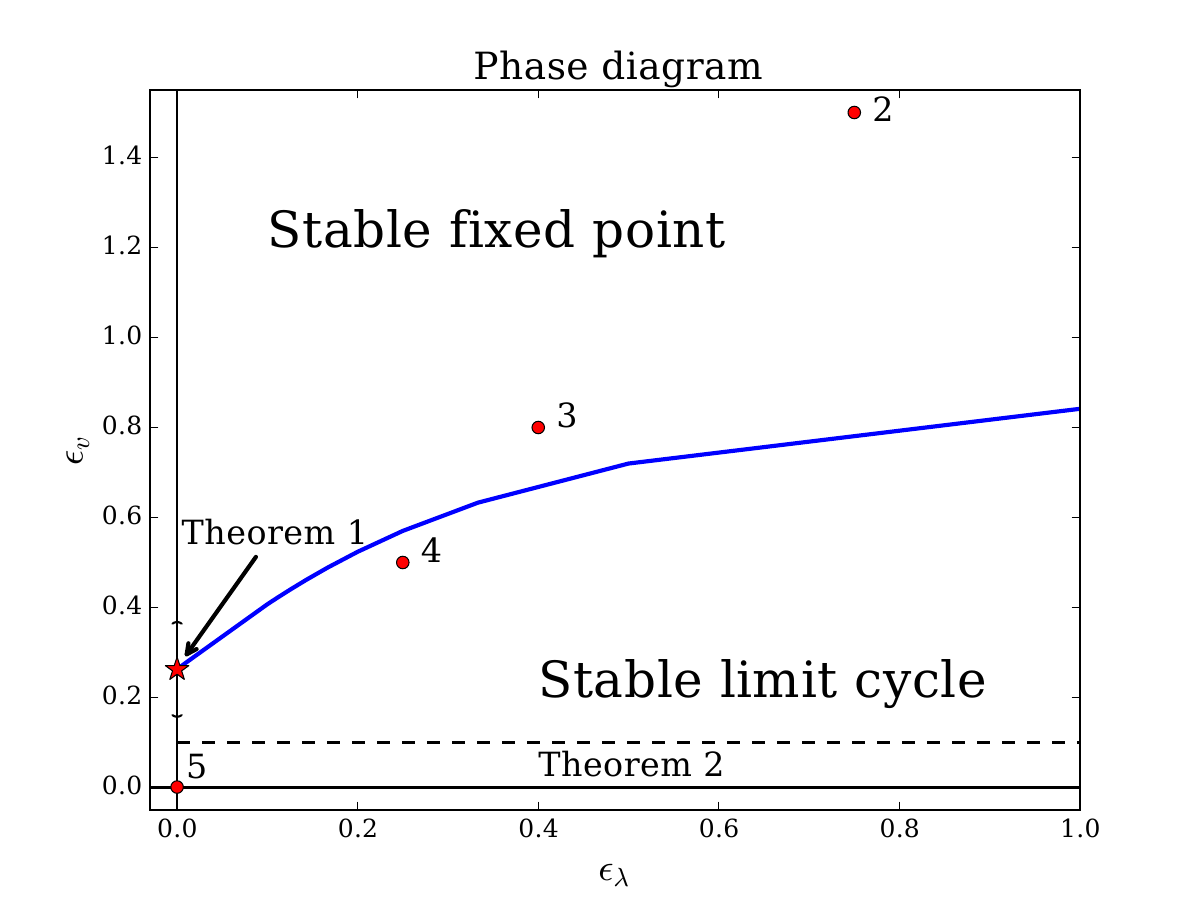}
\caption{In blue, we see the bifurcation value $\epsilon_v^*(\epsilon_{\lambda})$ numerically computed for a variety of values of $\epsilon_{\lambda}$. Simulations run with parameter values above the line exhibit a stable deadlock equilibrium, while values below the line exhibit oscillatory behavior. The red circles represent the values taken in the simulations displayed in succeeding figures with the corresponding number. Corroborating these numerical observations we establish the following formal results. In the single limit $\epsilon_\lambda \rightarrow 0$, Theorem \ref{thm:hopfResult} establishes a Hopf bifurcation at $\epsilon_{v,0} \approx 0.09175$, guaranteeing a family of stable limit cycles in a small (one-dimensional) neighborhood of $\epsilon_v$ values (at the $\epsilon_\lambda=0$ limit) around the red star. In the joint limit $\epsilon_\lambda,  \epsilon_v \rightarrow 0$, Theorem \ref{thm:mainResult} uses a singular perturbation argument to establish the persistence of stable limit cycles in some neighborhood of the abscissa of this plot.
}
\label{fig:bifurcationValue}
\end{center}
\end{figure}

Seeking formal confirmation of the limit cycles suggested by those simulations at the physically interesting parameter values where $\epsilon_\lambda > 0$, we next take recourse to a singular perturbation analysis. Specifically, we consider the joint limit $\epsilon_v \to 0, \epsilon_{\lambda} \to 0$ and carry out a dimension reduction of the system in this limit yielding planar dynamics exhibiting a limit cycle established by application of the Poincar\'e-Bendixson theorem. Arguments from geometric singular perturbation theory together with its conjectured (numerically corroborated) hyperbolicity then imply that this limit cycle persists for finite $\epsilon_v, \epsilon_{\lambda} > 0$.

This work is related to prior literature on dynamical decision-making in biological systems. Seeley \etal~\cite{TDS-etal:12} studied nest site selection behavior in honeybee swarms and discovered a mechanism called a stop signal, by which bees who were committed to one nest site physically wrestled bees committed to other sites in order to get them to abandon their commitment. Seeley \etal~constructed a dynamical systems model of this behavior and showed that the introduction of a stop signal allowed the system to avoid the deadlock state where no clear majority emerges in favor of any given option. We use the dynamical system from Seeley \etal~\cite{TDS-etal:12} which models value-based nest site selection in honeybee swarms to modulate the motivation state. We let the value associated with each task be modulated by how far the agent is from the goal state associated with that task. This introduces feedback into the motivation dynamics by making the current system state influence the task values and thereby the motivation state.

Pais \etal~\cite{DP-etal:13} studied Seeley \etal's model \cite{TDS-etal:12} using singular perturbation theory and showed that the stop signal also makes the model sensitive to the absolute value of the alternatives, allowing the system to remain in deadlock if all alternatives are equally poor. Pais \etal~suggested that this sensitivity is useful to avoid prematurely committing to a suboptimal alternative, and show that it results in hysteresis as a function of the difference in the value of the alternatives. More recently, Gray \etal~\cite{RG-etal:18} proposed a model of decision-making dynamics for networked agents inspired by the above-referenced works on nest site selection and showed that singularity theory \cite{MG-DGS:85} can provide a powerful tool to engineer desired outcomes in decision-making models. These convincing accounts of the utility and potential analytical tractability of such bioinspired decision models provide a direct point of departure for our work. Specifically, \cite{DP-etal:13, TDS-etal:12} studied one-off decisions where the value of each option (i.e., task) is static. In contrast, we allow the values of the tasks to change dynamically as they are completed by feeding back the system state, which allows the agent to determine the status of each task.

Other authors, particularly in the evolutionary dynamics literature, have studied systems with similar types of feedback. In evolutionary dynamics \cite{JH-KS:98}, which seeks to formalize Darwin's ideas about natural selection, a set of populations each representing different strategies interact with each other and the interaction determines the level of fitness of each strategy. Fit populations thrive and grow, while unfit populations die off. Pais, Caicedo, and Leonard \cite{DP-CHC-NEL:12} studied the replicator-mutator equations from evolutionary dynamics with a particular network structure to the fitness function and showed conditions under which the dynamics exhibit Hopf bifurcations resulting in limit cycles. Mitchener and Nowak \cite{WGM-MAN:04} studied evolutionary dynamics as a model of language transmission and showed conditions under which the dynamics of distinct grammars can exhibit limit cycles corresponding to periodic changes in the dominant grammar. The feedback model adopted by \cite{DP-CHC-NEL:12} and \cite{WGM-MAN:04} captures the evolutionary process in which the fitness of a given strategy is determined by the relative fractions of the population adopting that strategy. Such a model is inappropriate for our robotic application, where the value of a task need not arise from competitive interactions between tasks. Our Hopf analysis in Section \ref{sec:Hopf} is similar to that in \cite{DP-CHC-NEL:12}, but we go on to show the existence of limit cycles in a two-dimensional region in parameter space using tools from geometric singular perturbation theory.

The contributions of this paper are twofold. First, we lay out an intrinsically dynamical systems approach to the composition and prioritization of potentially competing tasks for mobile robots. This approach provides an alternative to existing approaches using temporal logic and suggests future work investigating the connections between logic-based and dynamical-systems-based approaches to modeling decision making. Second, we show how a bistable system with a pitchfork bifurcation (the motivation dynamics) can be incorporated into a feedback system and produce a system that exhibits a Hopf bifurcation. This result is intuitive, and, in particular, the results of \cite{MG-DGS:85} suggest interpreting the oscillation our system exhibits as arising from the feedback modulation of the unfolding parameters of the pitchfork embedded in the motivation dynamics. Elucidating this structure, e.g., by finding the normal form of the Hopf bifurcation embedded in \eqref{eq:dynamicsBaseCoordinates} is a nontrivial undertaking worthy of future study that may usefully inform further such designs.

The remainder of the paper is structured as follows. In Section \ref{sec:model} we lay out the broad class of systems under consideration before specifying the instance of the model which we study and stating our formal results. In Section \ref{sec:simulations} we show the result of several illustrative simulations, which suggest the existence of a Hopf bifurcation. In Section \ref{sec:Hopf} we study the system in the limit $\epsilon_{\lambda} \to 0$ and show the existence of a Hopf bifurcation that results in stable limit cycles. In Section \ref{sec:reduction} we study the system in the joint limit $\epsilon_{\lambda} \to 0, \epsilon_v \to 0$ and show the existence of a stable limit cycle in the resulting two-dimensional reduced system; in Section \ref{sec:persistence} we show that this limit cycle persists for finite values of $\epsilon_v$ and $\epsilon_{\lambda}$. Finally, we conclude in Section \ref{sec:conclusion}.

\section{Model, Problem Statement, and Formal Results}\label{sec:model}
In this section we define our system model, state the problem we address and the formal results we obtain.
\subsection{Model}
Our model consists of three interconnected dynamical subsystems: states representing the navigation tasks and associated control actions (vector fields); the motivation state $m$; and the value state $v$. Implicit in the definition of the navigation tasks is the definition of the physical agent, which comprises the agent's body and its workspace, or environment.

\subsubsection{Body, Environment, and Motivational States}
We model the robot as a point particle located at $x \in \domain$, where the environment $\domain \subseteq \bbR^d$ is a domain within Euclidean space. In general, $\domain$ may be punctured by obstacles, but in this initial work we restrict ourselves to unobstructed domains.

We represent \emph{motivation} by the state $m \in \Delta^{N}$, where $\Delta^N = \{ m \in \bbR^{N+1} : m_i \geq 0, \sum_{i=1}^{N+1} m_i = 1\}$ is the $N$-simplex. We index the first $N$ elements of $m$ by $i \in \oneto{N}$: $m_i$ represents the motivation to perform task $i$. The last element we label as $m_U$: this represents undecided motivation, i.e., the decision to not perform any task.

\subsubsection{Tasks}
The agent has a set of $N$ \emph{tasks}. Each task $i \in \oneto{N}$ requires navigating the agent to the location $x_i^* \in \domain$. Inspired by the navigation function framework \cite{ER-DEK:92}, we assume the existence of a distance function $\varphi_i: \domain \to \bbR_+$ for each task $i$. The function yields a gradient field $-\nabla \varphi_i$ such that $\dot{x} = -\nabla \varphi_i$ obeys
\[ \lim_{t \to +\infty} x(t) = x_i^*. \]
That is, the gradient field $F_i = -\nabla \varphi_i$ is a vector field that accomplishes task $i$. In the following, where the domain is assumed to be unobstructed, we define the navigation functions by the squared Euclidean distance 
\beq \label{eq:navFunction}
\varphi_i(x) = \frac{1}{2} \| x - x_i^* \|_2^2.
\eeq
 
Finally, we define the matrix-valued function consisting of the $N$ task navigation vector fields plus the null gradient field associated with indecision
\beq \label{eq:navFields}
\Phi(x) = \begin{bmatrix} F_1(x) & \hdots & F_N(x) & 0 \end{bmatrix} \in \bbR^{d \times (N+1)}.
\eeq
By taking convex combinations of these vector fields we can assign the agent weighted combinations of the instantaneous (``greedy'') task plans they represent; the motivation state, defined below, will specify the convex combination to be taken at any given time.

The agent's high-level mission is to repeatedly carry out each of the $N$ low-level tasks, i.e., visit each of the $N$ locations, in a specified order. In the vocabulary of the LTL hybrid systems literature, this corresponds to a recurrent patrol or coverage mission \cite{GEF-etal:09}. We now develop the detailed model, introducing its states and dynamics, then finally present statements of the problem we address and the formal results.

\subsection{Model Dynamics} \label{sec:dynamics}
Having specified the system model and its state space, we now define its dynamics. The system has state $(x, m, v) \in \domain \times \Delta^N \times \bbR_+^N$. The state variables evolve according to the dynamics
\begin{align}
\dot{x} &= f_x(x, m)\\
\dot{m} &= f_m(m, v)\\
\dot{v} &= f_v(v, x).
\end{align}
The specific forms of the functions $f_x, f_m,$ and $f_v$ are given in the following paragraphs.

\subsubsection{Navigation dynamics}
The body's location dynamics are the convex combination of the $N$ navigation vector fields (plus the null field associated with indecision), weighted by the motivation state:
\beq \label{eq:xDynamics}
\dot{x} = f_x(x,m) = -\Phi(x) m.
\eeq
For example, when $m = [1, 0, \cdots, 0]^\tran$, the navigation dynamics are $\dot{x} = -\Phi(x)m = -\nabla \varphi_1(x)$, and when $m = [0, \cdots, 0, 1]^\tran$, the dynamics are $\dot{x} = 0.$

\subsubsection{Motivation dynamics}
We take the motivation state dynamics from Pais \emph{et al.}'s work \cite{DP-etal:13} studying group decision making behavior in honeybee swarms:
\begin{align} \label{eq:miDynamics} 
\dot{m}_i = \tilde{v}_i m_U - m_i \left( 1/\tilde{v}_i - \tilde{v}_i m_U + \sigma_i (1 - m_i - m_U) \right).
\end{align}
We set $\tilde{v}_i = v^*_i v_i$, where $v_i \in \bbR_+$ is the value of task $i$ and $v_i^* > 0$ is a gain parameter that scales $v_i$. Equation \eqref{eq:miDynamics} holds for each $i \in \oneto{N}$, with the dynamics for $m_U$ following from the constraint that defines the simplex.

The dynamics \eqref{eq:miDynamics} were derived for group decision making in \cite{TDS-etal:12} from a microscopic individual-level Markov process model that incorporates commitment, abandonment, recruitment, and stop signal mechanisms. The term $\tilde{v}_i m_U$ represents spontaneous commitment of an uncommitted individual to option $i$ at a rate which is proportional to the value $\tilde{v}_i$, $-m_i/\tilde{v}_i$ represents spontaneous abandonment, $\tilde{v}_i m_i m_U$ represents recruitment of an uncommitted individual by one committed to option $i$, and $-\sigma_i m_i (1-m_i-m_U)$ represents a signal from individuals committed to options other than $i$ telling individuals committed to option $i$ to abandon their commitment. In our context where $m$ represents a single decision maker's motivation state, each of these mechanisms can be interpreted as modeling specific processes between neurons in the decision maker's brain rather than between individuals in a group.

\subsubsection{Value dynamics}
We define the dynamics of $v_i \in \bbR_+$, the value associated with state $i$, by
\begin{align}
\dot{v}_i 
	&= \lambda_i (\varphi_i(x) - v_i), i \in \oneto{N}. \label{eq:viDynamics}
\end{align}
Recall that the distance functions $\varphi_i$ take the value zero at the goal and increase with distance from the goal $x_i^*$. The intuition behind these dynamics is that the value (i.e., urgency) of a task should increase when the agent is far from that task's goal state and decrease when the agent reaches that state; a larger $v_i$ means that task $i$ is more valuable, i.e., more urgent. We encode this intuition in the linear time-invariant dynamics \eqref{eq:viDynamics}, which cause $v_i$ to follow $\varphi_i(x)$ with a lag associated with the integration time scale $\lambda_i$. The dynamics correspond closely to the concept of drive reduction theory in social psychology, where motivation is thought to arise from the desire to carry out actions that satisfy various intrinsic drives \cite{RSW:18, CLH:43}.

\subsection{Formal Problem Statement and Analytical Results}
The foregoing presentation introduces a broad class of models whose application to specific problems of reactive task planning and motivational control of multiple competing tasks we intend to explore empirically on physical robots. For the analytical purposes of this paper we find it expedient to consider a severely restricted instance from that class entailing only two, greatly simplified tasks and affording, in turn, a low-dimensional parametrization through imposition of various symmetries. In this section we first introduce the details of that restricted problem class and then state the analytical results we obtain.

\subsubsection{Two Tasks, Their Essential Parameters, and New Coordinates}
We have four parameters for each task $i \in \oneto{N}$. Each task requires navigating to a goal location $x_i^* \in \domain$. In the motivation dynamics \eqref{eq:miDynamics}, there is a positive stop signal parameter $\sigma_i > 0$ and value gain $v_i^* > 0$. Finally, in the value dynamics \eqref{eq:viDynamics}, there is a time scale $\lambda_i > 0$. We show that the number of parameters can be greatly reduced and that the system's behavior can be largely understood by varying the value of $v_i^*$.

For many parameter values, the system \eqref{eq:xDynamics}--\eqref{eq:viDynamics} exhibits a stable limit cycle in numerical simulations. To systematically study the system, we specialize to the case of a planar workspace $\domain = \bbR^2$ and $N=2$ tasks. Then the state space of the system \eqref{eq:xDynamics}--\eqref{eq:viDynamics} is $\bbR^2 \times \Delta^2 \times \bbR_+^2$, for which we pick the coordinates $\xi = (x_1, x_2, m_1, m_2, v_1, v_2)$. Furthermore, we set the following parameter values.

We set the two goal locations to $x_1^* = (1, 0) \in \domain, x_2^* = (-1, 0) \in \domain$. This choice is made without loss of generality, as it amounts to a translation, rotation, and scaling of the coordinates $x$ for $\domain$. Scaling the coordinates $x$ by a factor $\alpha$ requires scaling the $v$ coordinates and parameters $v_i^*$ by a factor $\alpha^2$ to account for the fact that $\varphi_i(x)$ is homogenous of degree 2. For the stop signal $\sigma_i$, we follow Pais \emph{et al.} \cite{DP-etal:13} and impose the symmetry $\sigma_1 = \sigma_2 = \sigma$. Similarly, for ease of exposition and analysis we equate the value gain parameters $v_1^* = v_2^* = v^* > 0$ as well as the value time scale parameters $\lambda_1 = \lambda_2 = \lambda > 0$. With these choices the set of system parameters is reduced to $\sigma, v^*,$ and $\lambda$, each of which must be positive. Fixing $\sigma$ at a nominal value, e.g., 8, leaves $v^*$ and $\lambda$ as free parameters whose values determine the behavior of the system.

The symmetry $\sigma_1 = \sigma_2$ makes the decision-making mechanism, i.e., the motivation dynamics \eqref{eq:miDynamics}, obey an $S_2$ symmetry (i.e., symmetry under the interchange of the task numbers $1 \leftrightarrow 2$) when the two task values are equal, i.e., $\tilde{v}_1 = \tilde{v}_2$. The concrete interpretation of this symmetry is that the decision-making mechanism has no inherent bias for either underlying task. We find this characteristic to be desirable because it means that biases, if desired, can be introduced through the task values $v$. Isolating the biases in this way simplifies the analysis and eventual programming of the resulting system.\footnote{In the case where $\sigma_1 = \sigma_2$ and $v^*_1 = v^*_2$, the overall feedback system is $S_2$ invariant, though this may not be desirable in an application where the underlying navigation tasks are not of equal importance. For such an application, our formulation makes it natural to encode the difference in importance by breaking the symmetry $v_1^* = v_2^*$.} The $S_N$ symmetry requirement naturally suggests further work in dynamical systems theory to study the dynamics of systems on $\Delta^N$ with the $S_N$ symmetry. Previous authors, e.g, \cite{MG-IS-DGS:88}, have studied dynamical systems using the symmetry approach, and some recent work, e.g., \cite{RG-etal:18}, has begun to use this theory in engineering applications; the present work suggests another application area for this theory.

Breaking the $S_2$ symmetry by allowing $\sigma_1$ and $\sigma_2$ to be set independently would introduce a bias towards one or the other options in the decision mechanism. Similarly, allowing $\lambda_1$ and $\lambda_2$ to be set independently also introduces a bias, this time in the valuation. Numerical experimentation suggests that the limit cycle behavior analyzed with the assumed symmetries persists for asymmetric settings of the parameters. Pais \etal~\cite{DP-etal:13} extensively studied the motivation dynamics in isolation and showed that, for fixed $v$, the dynamics \eqref{eq:miDynamics} exhibit a pitchfork bifurcation as $\sigma=\sigma_1=\sigma_2$ is increased through a threshold value. As $\sigma$ is increased, the oscillations become increasingly nonlinear (as can be expected since $\sigma$ scales a nonlinearity in the dynamics \eqref{eq:miDynamics}). In the limit of large $\sigma$ the dynamics tend towards a relaxation oscillation where the change in the motivation state $m$ happens quickly relative to the change in location state $x$. All evidence suggests that the parameters must be such that the motivation dynamics are in the post-bifurcation regime for the limit cycle to exist, however we will make careful analysis of the connection between the pitchfork bifurcation of the motivation dynamics and the Hopf bifurcation of the closed-loop system the subject of future work.

In the case $N=2$ and assuming $\varphi_i$ defined by \eqref{eq:navFunction}, the equations \eqref{eq:xDynamics}--\eqref{eq:viDynamics} are
\beq \label{eq:dynamicsBaseCoordinates}
\dot \xi = f_{\xi}(\xi),
\eeq
where the components of $f_{\xi}$ are given by
\begin{align*}
\dot{x} &= -m_1 (x-x_1^*) - m_2 (x-x_2^*)\\
\dot{m}_1 &= (v^* v_1) m_U - m_1 (1/(v^* v_1) - (v^* v_1) m_U + \sigma m_2)\\
\dot{m}_2 &= (v^* v_2) m_U - m_2 (1/(v^* v_2) - (v^* v_2) m_U + \sigma m_1)\\
\dot{v}_1 &= \lambda (\varphi_1(x) - v_1)\\
\dot{v}_2 &= \lambda (\varphi_2(x) - v_2).
\end{align*}

We change coordinates on $\Delta^2 \times \bbR^2_+$ by transforming into mean and difference coordinates defined by 
\[ \delm = m_1 - m_2, \mbar = \frac{m_1 + m_2}{2}; \ \delv = v_1 - v_2, \vbar = \frac{v_1 + v_2}{2}. \]
Define the coordinates 
\beq \label{eq:defZ}
z = (x_1, x_2, \delm, \mbar, \delv, \vbar)
\eeq 
on the space $\bbR^2 \times \Delta^2 \times \bbR \times \bbR_+$ and parameters $\epsilon_v = 1/v^*$ and $\epsilon_{\lambda} = 1/\lambda$. It is easy to see that the transformation from $\zeta$ to $z$ is a diffeomorphism. In the mean-difference coordinates, the dynamics \eqref{eq:dynamicsBaseCoordinates} are
\beq \label{eq:mean-differenceDynamics}
\dot z = f_z(z),
\eeq
where the components of $f_z$ are given by
\begin{align} \label{eq:x1Dot}
\dot{x}_1 = f_{x_1}(z) = &\delm -2 \mbar x_1,
\end{align}

\begin{align} \label{eq:x2Dot}
\dot{x}_2 = f_{x_2}(z) =&-2\mbar x_2,
\end{align}

\begin{align} \label{eq:DeltamDot}
\dot{\Delta m} = f_{\delm}(z) = - &\epsilon_v \left( \frac{2 \bar{m} + \Delta m}{2 \vbar + \delv} -\frac{2 \bar{m} - \Delta m}{2 \vbar - \delv} \right)\\ 
&+ \vbar \Delta m (1-2 \bar{m})/\epsilon_v + \delv (1-2 \bar{m}) (1+ \bar{m})/\epsilon_v, \nonumber
\end{align}

\begin{align} \label{eq:mBarDot}
\dot{\bar{m}} = f_{\mbar}(z) = \frac{1}{2} &\biggl( - \epsilon_v \frac{2 \bar{m} + \Delta m}{2 \vbar + \delv} - \epsilon_v \frac{2 \bar{m} - \Delta m}{2 \vbar - \delv} \\ \nonumber
 &+ \frac{2 \vbar + \delv}{2 \epsilon_v}(1-2\bar{m}) (1 + \frac{2 \bar{m} + \Delta m}{2})\\ \nonumber
 &+ \frac{2 \vbar - \delv}{2 \epsilon_v}(1-2\bar{m}) (1 + \frac{2 \bar{m} - \Delta m}{2})\\ \nonumber
 &  - \frac{\sigma}{2} (2 \bar{m} + \Delta m)(2 \bar{m} - \Delta m) \biggr) ,
\end{align}

\beq \label{eq:DeltavDot}
\epsilon_{\lambda} \dot{\delv} = f_{\delv}(z) = -(\delv - \Delta \varphi(x)),
\eeq

\beq \label{eq:vBarDot}
\epsilon_{\lambda} \dot{\vbar} = f_{\vbar}(z) = -(\vbar - \bar{\varphi}(x)),
\eeq
where $\delp(x) = \varphi_1(x) - \varphi_2(x), \pbar(x) = (\varphi_1(x) + \varphi_2(x))/2$.

Note that the system \eqref{eq:dynamicsBaseCoordinates} exhibits a symmetric equilibrium state which we call a (symmetric) deadlock equilibrium, adopting the nomenclature established in the literature, e.g., \cite{DP-etal:13}. The word \emph{deadlock} refers to the fact that in such an equilibrium, the system fails to reach a decision. Precisely, we have the following:
\begin{definition}
Let $\xi_d$ be an equilibrium of the dynamics \eqref{eq:dynamicsBaseCoordinates} of the form $\xi_d = (x_{1,d}, x_{2,d}, m_d, m_d, v_d, v_d)$. Then $\xi_d$ is called a \emph{(symmetric) deadlock equilibrium}. Equivalently, in $z$ coordinates, an equilibrium $z_d$ of the dynamics \eqref{eq:mean-differenceDynamics} is called a (symmetric) deadlock equilibrium if it is of the form $z_d = (x_{1,d}, x_{2,d}, 0, \mbar_d, 0, \vbar_d)$.
\end{definition}
It is straightforward to see that a symmetric deadlock equilibrium exists if there is a point $x_d \in \domain$ such that $\varphi_1(x_d) = \varphi_2(x_d) = \pbar$.

\subsubsection{Formal Results} \label{sec:formalResults}
In this section we state the two theorems which constitute the formal results of the paper. As can be seen in the simulations presented in Section \ref{sec:simulations}, the dynamics \eqref{eq:mean-differenceDynamics} appear to exhibit a Hopf bifurcation as the parameters $\epsilon_v$ and $\epsilon_{\lambda}$ approach zero, giving birth to stable limit cycles. We formalize this observation in two steps. First we consider the limit $\epsilon_\lambda \to 0$ which reduces the dimension of the system \eqref{eq:mean-differenceDynamics} and permits an explicit computation showing the existence of a Hopf bifurcation.

In the limit $\epsilon_{\lambda} \to 0$, the $v$ variables are directly coupled to $\varphi(x)$, so $\delv = \delp(x)$ and $\vbar = \pbar(x)$, which are fixed points of Equations \eqref{eq:DeltavDot} and \eqref{eq:vBarDot}, respectively. Define $z_r = (x_1, x_2, \delm, \mbar)$ as the vector of the remaining state variables. Explicitly, $z$ and $z_r$ are related by the linear embedding $z = h(z_r)$ with left inverse given by the linear projection $h^{\dagger}$, where
\[ h(z_{r,1}, z_{r,2},z_{r,3},z_{r,4}) :=
(z_{r,1}, z_{r,2},z_{r,3},z_{r,4}, z_{r,1}, z_{r,2}); \quad
h^\dagger(z) := (z_1, z_2, z_3, z_4).
\]
Then the dynamics \eqref{eq:mean-differenceDynamics} reduce to the restriction dynamics
\beq 
\dot{z}_r = f_r(z_r, \epsilon_v) := D h^{\dagger} \cdot f_z \circ h(z_r), \label{eq:dynamicsDirectFeedback}
\eeq
The restriction dynamics exhibit a Hopf bifurcation, as summarized in the following theorem:

\begin{theorem} \label{thm:hopfResult}
The system $\dot z_r = f_r(z_r, \epsilon_v)$ defined by \eqref{eq:dynamicsDirectFeedback} has a deadlock equilibrium $z_{rd}$ given by \eqref{eq:reducedDeadlockEq}. For $\sigma > 6$, the dynamics undergo a Hopf bifurcation resulting in stable periodic solutions at $(z_{rd}, \epsilon_{v,0})$, where $\epsilon_{v,0}$ is the unique solution $\epsilon_v \in [0, 1/2]$ of $4(2-\sigma) \epsilon_v^2 - 4(2-\sigma) \epsilon_v + (6-\sigma) = 0$. For $\sigma = 8, \epsilon_{v,0} = (3-\sqrt{6})/6 \approx 0.09175$.
\end{theorem}

As we are ultimately motivated by the physically meaningful case of small but non-zero values of $\epsilon_v$ and $\epsilon_\lambda$, we study the singular perturbation limit $\epsilon_v, \epsilon_\lambda \to 0$ under which the system \eqref{eq:mean-differenceDynamics} can be reduced to a planar dynamical system and show the existence of a limit cycle. We then employ tools from geometric singular perturbation theory to show the persistence of this limit cycle for sufficiently small, but finite, values of $\epsilon_v$ and $\epsilon_\lambda$:
\begin{theorem} \label{thm:mainResult}
Accepting Conjecture \ref{conj:Floquet}, below, for $\sigma>6$, there exists a stable limit cycle of \eqref{eq:mean-differenceDynamics} for sufficiently small, but finite, values of $\epsilon_{\lambda}$ and $\epsilon_v$. Equivalently, fixing $\lambda$, there exists a stable limit cycle of \eqref{eq:mean-differenceDynamics} for sufficiently large, but finite, values of $v^*$.
\end{theorem}

\section{Illustrative Simulations}\label{sec:simulations}

Figure \ref{fig:bifurcationValue} summarizes the behavior of the system \eqref{eq:mean-differenceDynamics} as a function of the two parameters $\epsilon_v$ and $\epsilon_{\lambda}$. For large values of both parameters, the system exhibits a stable deadlock equilibrium, while for sufficiently small values of both parameters the system exhibits a stable limit cycle. As $\sigma$ is increased, the limit cycle tends towards a structure composed of a slow segment followed by a fast jump, which is characteristic of relaxation oscillations \cite{SMB-TE:86,JG:04}. Section \ref{sec:Hopf} studies the system in the limit $\epsilon_{\lambda} \to 0$ and analytically shows the existence of a Hopf bifurcation at $\epsilon_v = \epsilon_{v,0}$ ($\approx 0.09175$ for $\sigma = 8$). The blue line in Figure \ref{fig:bifurcationValue} shows the numerically-computed bifurcation value $\epsilon^*_v( \epsilon_{\lambda})$ for $\epsilon_{\lambda} > 0$. The numerically-computed limit $\lim_{\epsilon_\lambda \to 0} \epsilon^*_v(\epsilon_\lambda)$ corresponds well to the analytical value $\epsilon_{v,0}$.

Figures \ref{fig:pre-bifurcation}--\ref{fig:DualLimit} show simulations of the system \eqref{eq:mean-differenceDynamics} for four representative values of the parameters $\epsilon_v, \epsilon_{\lambda}$. We set $\sigma = 8$. In $\xi$ coordinates, the initial conditions were $x = 0, m_1 = 0, m_2 = 1/2, v_1 = v_2 = 0.1$. In the mean-difference coordinates $z$ this corresponds to $\delm=-1/2, \mbar=1/4, \delv = 0$, and $\vbar = 0.1$. This choice of initial conditions was made to avoid the deadlock equilibrium but was otherwise generic.

\begin{figure}[htbp]
\begin{center}
\includegraphics[width=3.5in]{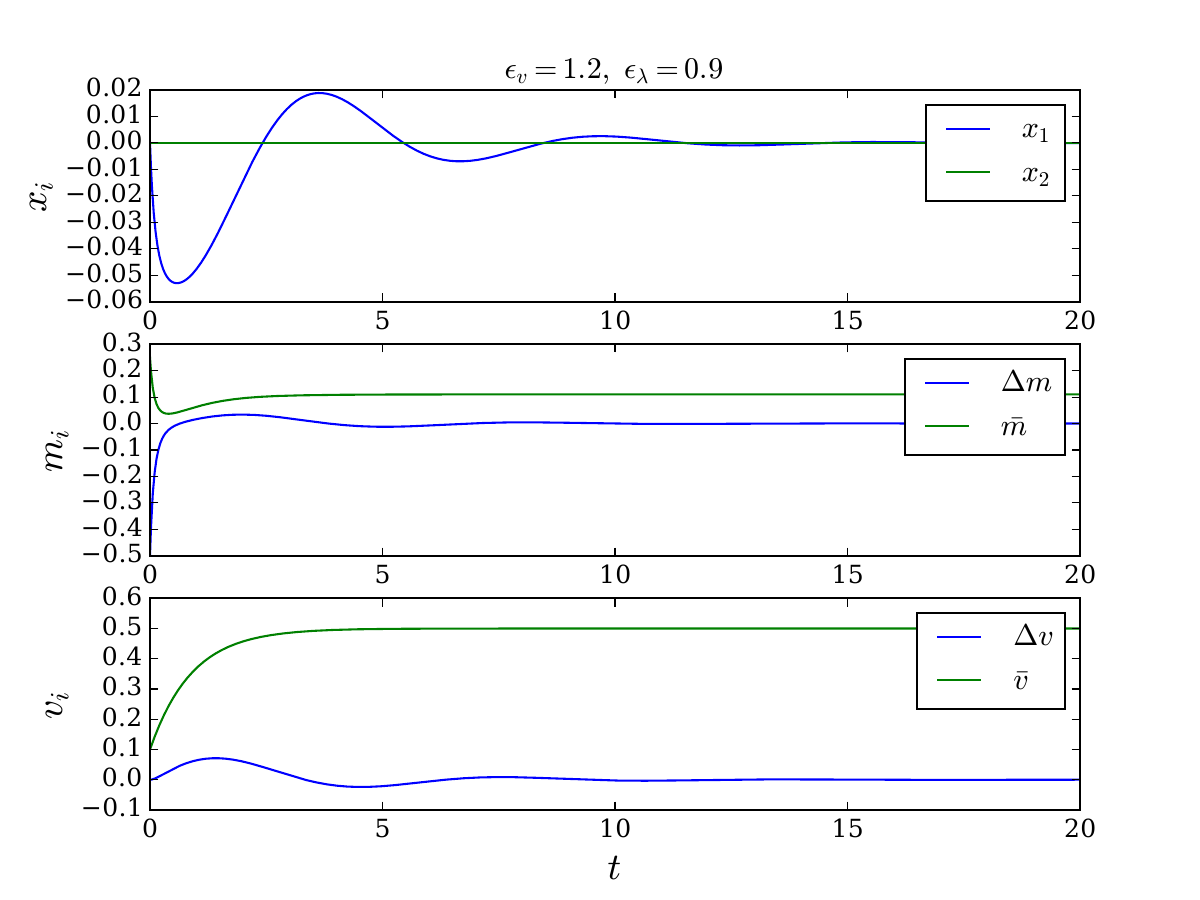}
\caption{Large values of both scales, represented by $\epsilon_v = 1.2, \epsilon_{\lambda} = 0.9$. The system converges to a stable deadlock equilibrium with minimal oscillation.}
\label{fig:pre-bifurcation}
\end{center}
\end{figure}

\begin{figure}[htbp]
\begin{center}
\includegraphics[width=3.5in]{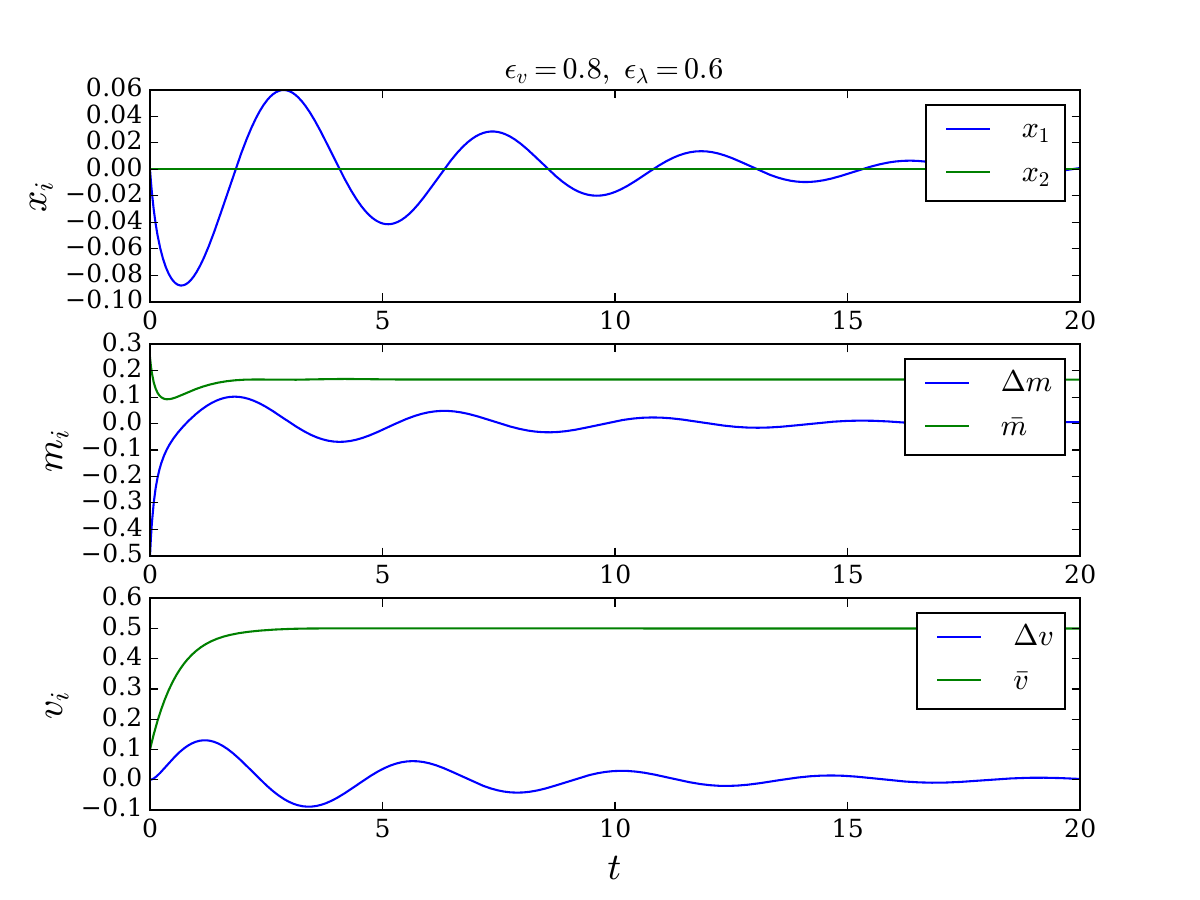}
\caption{Parameter values $\epsilon_v = 0.8, \epsilon_{\lambda} = 0.6$ that are near the Hopf bifurcation but still in the stable fixed point regime. The system displays damped oscillatory behavior that appears nearly linear, as to be expected near a Hopf bifurcation.}
\label{fig:dampedOscillation}
\end{center}
\end{figure}

\begin{figure}[htbp]
\begin{center}
\includegraphics[width=3.5in]{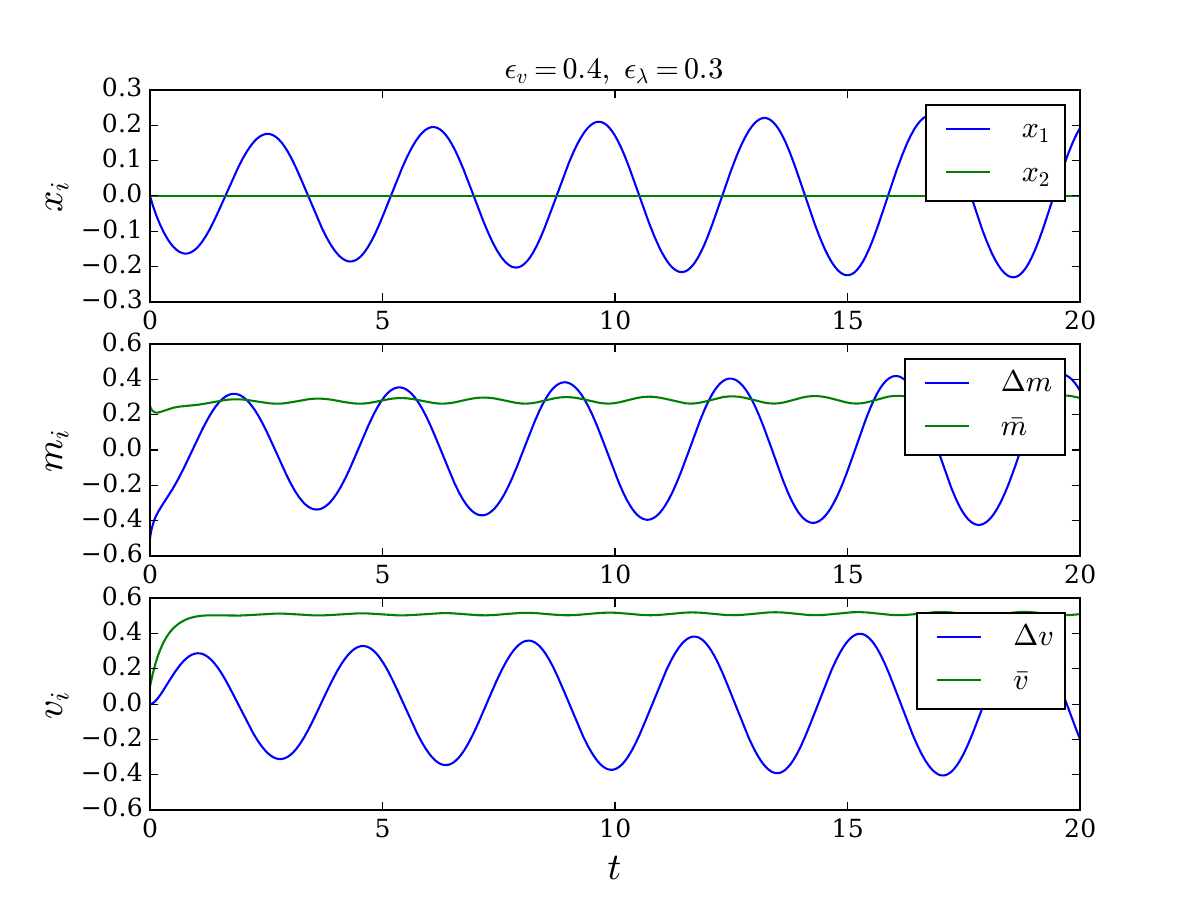}
\caption{Parameter values $\epsilon_v = 0.4, \epsilon_{\lambda} = 0.3$ that are near the Hopf bifurcation in the stable limit cycle regime. The system settles down to roughly ``harmonic'' oscillatory behavior whose (nearly) linear appearance is consistent with its proximity to the Hopf bifurcation.}
\label{fig:smallLimitCycle}
\end{center}
\end{figure}

\begin{figure}[htbp]
\begin{center}
\includegraphics[width=3.5in]{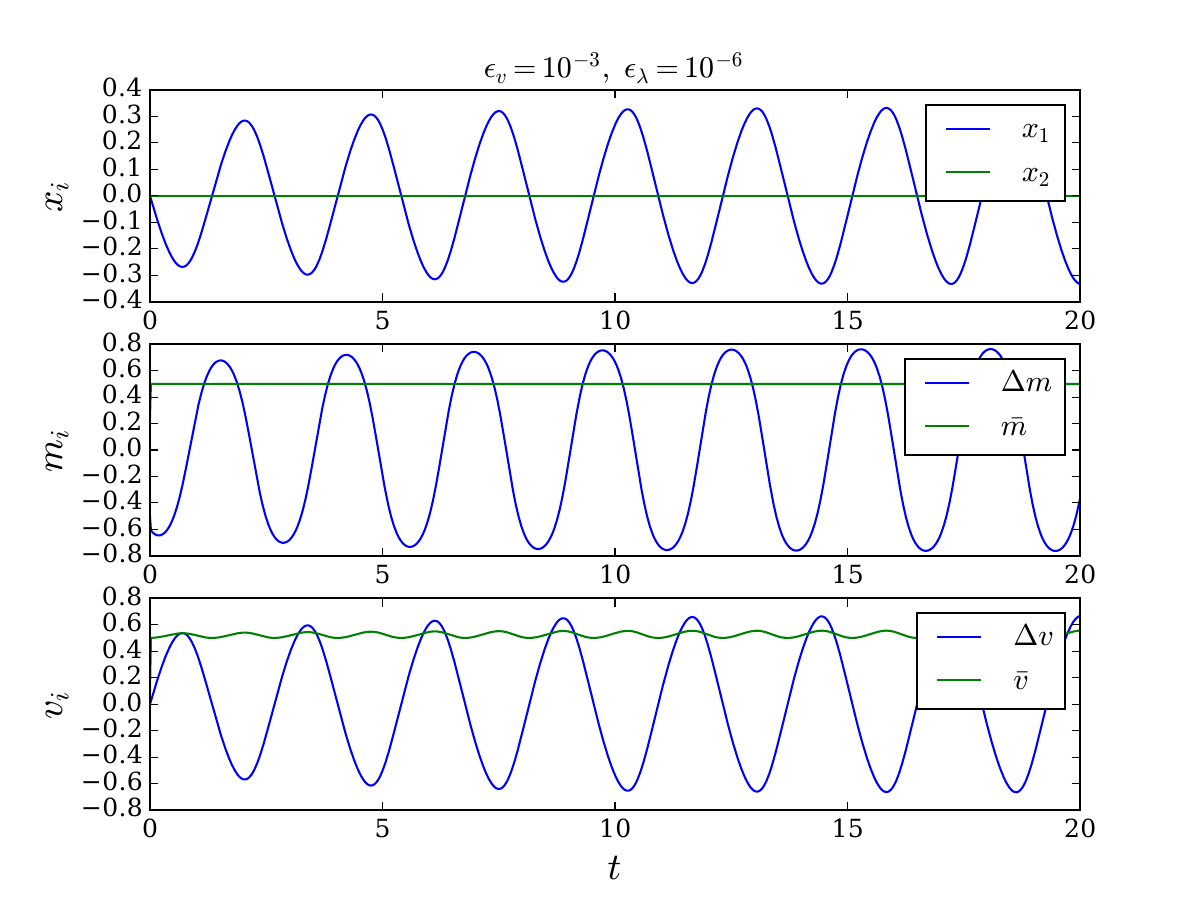}
\caption{Both parameters $\epsilon_v$, and $\epsilon_{\lambda}$ taken to be small, represented by $\epsilon_v = 10^{-3}, \epsilon_{\lambda} = 10^{-6}$. The system displays oscillatory behavior that appears weakly nonlinear. The $v$ variables are tightly coupled to the $x$ variables due to the small value of $\epsilon_{\lambda}$ in Equations \eqref{eq:DeltavDot} and \eqref{eq:vBarDot}. The variables $x_1, \delm,$ and $\delv$ oscillate in a coupled way, with $\delm$ oscillating in a nonlinear manner. Also note that the variables $x_2, \mbar,$ and $\vbar$ appear to stably converge to limiting values: this, combined with the coupling between $x_1$ and $\delv$, strongly suggests that the dynamics can be reduced to a two-dimensional system.}
\label{fig:DualLimit}
\end{center}
\end{figure}

Figure \ref{fig:pre-bifurcation} suggests that for large values of $\epsilon_v, \epsilon_{\lambda}$ there is a stable deadlock equilibrium in the system where no oscillations are present. Analyzing the dynamics, we see that this equilibrium corresponds to the state $z_d = (x_{1,d}, x_{2,d}, \delm_d, \mbar_d, \delv_d, \vbar_d)$, where $x_{1,d} = x_{2,d} = 0, \delm_d = \delp_d = 0, \vbar_d = 1/2$, and $\mbar_d$ solves the following quadratic equation:
\beq \label{eq:mBarDeadlockEqn}
- 2 (1 + \sigma \epsilon_v) \mbar_d^2 - (4 \epsilon_v^2 + 1) \mbar_d + 1 = 0,
\eeq
which has the solution $\mbar_d(\epsilon_v)$ given by
\begin{align} \label{eq:mBarDeadlock}
&\frac{-(4 \epsilon_v^2 + 1) + \sqrt{(4 \epsilon_v^2 + 1)^2 + 8 (1 + \sigma \epsilon_v)} }{4 (1 + \sigma \epsilon_v)}\\
= &\frac{-(4 \epsilon_v^2 + 1) + \sqrt{16 \epsilon_v^4 + 8 \epsilon_v^2 + 8 \sigma \epsilon_v +9 } }{4 (1 + \sigma \epsilon_v)}, \nonumber
\end{align}
which is clearly positive, as $\sigma$ and $\epsilon_v$ are both positive, which implies that the second term under the radical in \eqref{eq:mBarDeadlock} is positive. Figures \ref{fig:pre-bifurcation}--\ref{fig:DualLimit} suggest that the system undergoes a Hopf bifurcation as the parameters $\epsilon_v$ and $\epsilon_{\lambda}$ are decreased. In the following sections we carry out a series of analyses to characterize the bifurcation and study the resulting limit cycle.

\section{Hopf analysis in the limit $\epsilon_{\lambda} \to 0$}\label{sec:Hopf}
Motivated by the numerical evidence of a Hopf bifurcation occurring at the deadlock equilibrium, we consider the system \eqref{eq:mean-differenceDynamics} in the limit $\epsilon_{\lambda} \to 0$ and analytically show the existence of a Hopf bifurcation in this limiting case as $\epsilon_v$ is lowered through a critical value $\epsilon_{v,0}$. We then numerically consider the case of finite $\epsilon_{\lambda}$ and compute the bifurcation value $\epsilon_v^*(\epsilon_{\lambda})$ for a range of values of $\epsilon_{\lambda}$; the numerically-computed limit $\lim_{\epsilon_{\lambda} \to 0} \epsilon_v^*(\epsilon_{\lambda})$ matches the analytical result $\epsilon_{v,0}$, as shown in Figure \ref{fig:bifurcationValue}.

\subsection{Dynamics in the limit $\epsilon_{\lambda} \to 0$}
The limit dynamics \eqref{eq:dynamicsDirectFeedback} inherits the deadlock equilibrium $z_{rd} := h^\dagger(z_d)$ from the full dynamics \eqref{eq:mean-differenceDynamics}, where 
\beq \label{eq:reducedDeadlockEq}
z_{rd} := h^\dagger(z_d) = (x_{1rd}, x_{2rd}, \delm_{rd}, \mbar_{rd}),
\eeq
$x_{1rd} = x_{2rd} = \delm_{rd} = 0$, and $\mbar_{rd}$ again solves Equation \eqref{eq:mBarDeadlockEqn}.

In Figure \ref{fig:bifurcationDiagram} we show the numerically-computed bifurcation diagram for the system \eqref{eq:dynamicsDirectFeedback} with bifurcation parameter $\epsilon_v$. For large values of $\epsilon_v$, the deadlock equilibrium is stable. As $\epsilon_v$ is lowered below the critical value $\epsilon_{v,0}$, the system undergoes a Hopf bifurcation that results in a limit cycle. In Figure \ref{fig:bifurcationDiagram}, we plot the amplitude of the oscillations of $\delp$ for the limit cycle. As can be seen from Equation \eqref{eq:defZ}, $\delp$ is constrained to take values in $[-1, 1]$, so the limit cycle's amplitude is bounded above by $1$.

\begin{figure}[htbp]
\begin{center}
\includegraphics[width=3.5in]{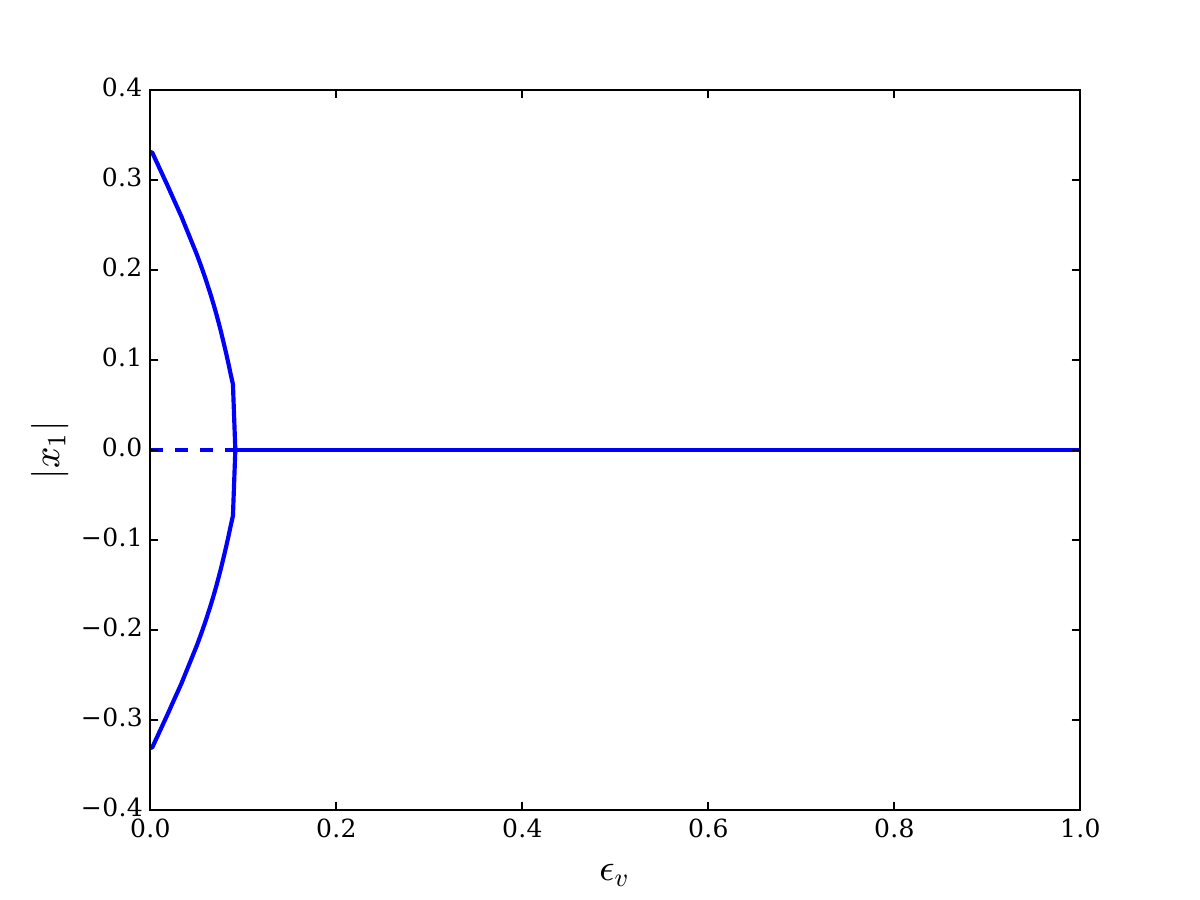}
\caption{The numerically-computed bifurcation diagram for the system \eqref{eq:dynamicsDirectFeedback} with $\epsilon_{\lambda} \to 0$ and bifurcation parameter $\epsilon_v$. The amplitude of the limit cycle is computed as the amplitude of the oscillations in $x_1$. We clearly see a supercritical Hopf bifurcation, with bifurcation value $\epsilon_{v,0}$. The stop signal parameter was set to $\sigma=8$.}
\label{fig:bifurcationDiagram}
\end{center}
\end{figure}

\begin{figure}[htbp]
\begin{center}
\includegraphics[width=3.5in]{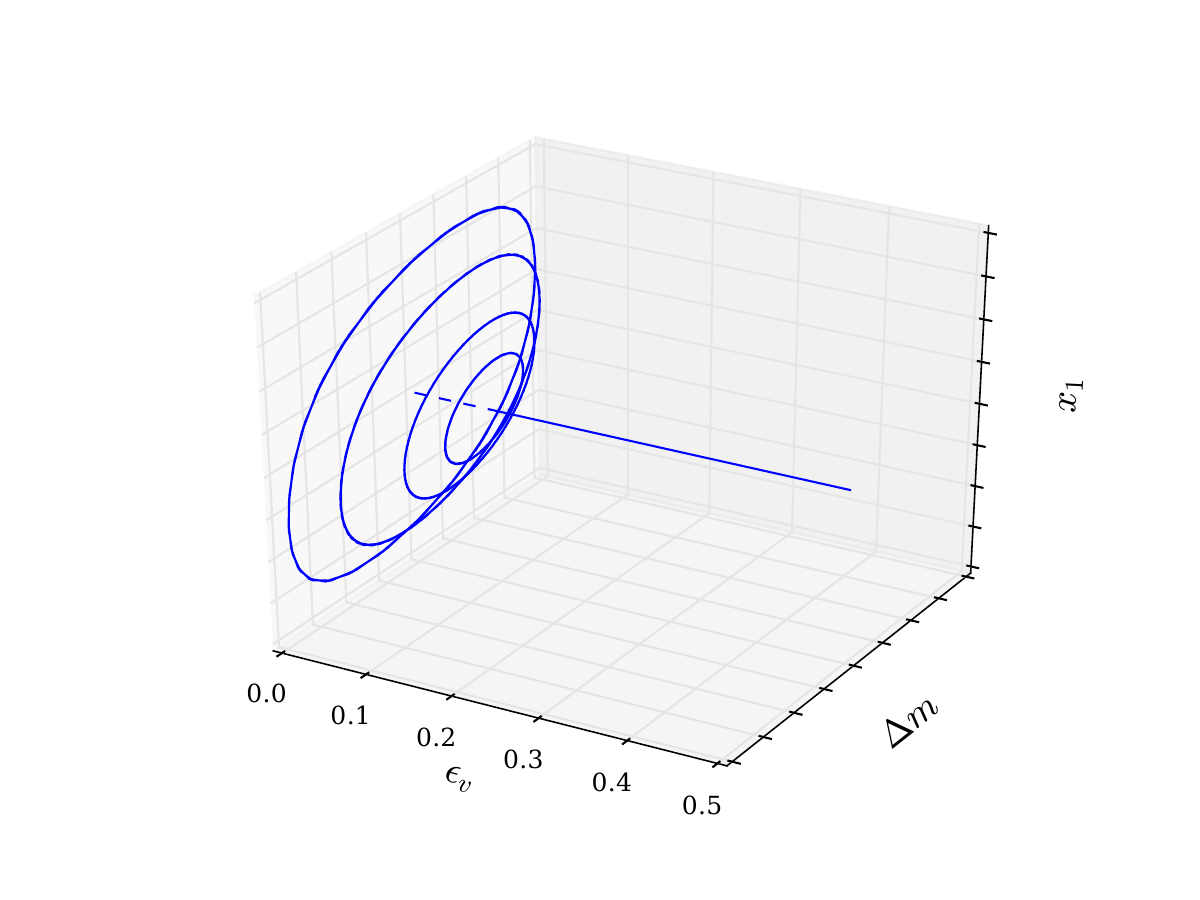}
\caption{The numerically-computed bifurcation diagram for the system \eqref{eq:dynamicsDirectFeedback} with $\epsilon_{\lambda} \to 0$ and bifurcation parameter $\epsilon_v$. The $x$-axis is the bifurcation parameter $\epsilon_v$, while the $y$ and $z$-axes are $\delm$ and $x_1$, respectively. As the bifurcation parameter is lowered below $\epsilon_{v,0}$, the deadlock equilibrium becomes unstable and gives birth to a limit cycle whose amplitude grows as $\epsilon_v \to 0$. The stop signal parameter was set to $\sigma=8$.}
\label{fig:bifurcationDiagram3d}
\end{center}
\end{figure}

\subsection{Analysis of the $\epsilon_{\lambda} \to 0$ dynamics} \label{sec:hopfBifurcation}
Inspired by the bifurcation diagram, we now seek to show the existence of the Hopf bifurcation suggested by Figure \ref{fig:bifurcationDiagram}. The following theorem from \cite{JG-PJH:13} summarizes the conditions under which a system undergoes Hopf bifurcation.
\begin{theorem}[Hopf bifurcation, {\cite[Theorem 3.4.2]{JG-PJH:13}}] \label{thm:hopfThm}
Suppose that the system $\dot z = f(z, \mu), z \in \bbR^n, \mu \in \bbR$, has an equilibrium $(z_0, \mu_0)$ and the following properties are satisfied:
\begin{enumerate}
\item The Jacobian $\left. D_z f \right|_{(z_0, \mu_0)}$ has a simple pair of pure imaginary eigenvalues
$\lambda(\mu_0)$ and $\bar \lambda (\mu_0)$ and no other eigenvalues with zero real parts,
\item $\left. \mathrm{d} (\mathrm{Re}\ \lambda(\mu))/ \mathrm{d} \mu \right|_{\mu = \mu_0} = d \neq 0.$
\end{enumerate}
Property 1) implies that there is a smooth curve of equilibria $(z(\mu), \mu)$ with $z(\mu_0) = z_0$. The eigenvalues $\lambda(\mu), \bar{\lambda}(\mu)$ of $\left. D_z f \right|_{(z(\mu_0), \mu_0)}$ which are imaginary at $\mu = \mu_0$ vary smoothly with $\mu$.

If Property 2) is satisfied, then there is a unique three-dimensional center manifold passing through $(z_0, \mu_0)$ in $\bbR^n \times \bbR$ and a smooth system of coordinates (preserving the planes $\mu$=const.) for which the Taylor expansion of degree 3 on the center manifold is given by \cite[(3.4.8)]{JG-PJH:13}. If $\left. \ell_1 \right|_{(z_0, \mu_0)} \neq 0$, there is a surface of periodic solutions in the center manifold which has quadratic tangency with the eigenspace of $\lambda(\mu_0), \bar{\lambda}(\mu_0)$ agreeing to second order with the parabaloid $\mu = -(\left. \ell_1 \right|_{(z_0, \mu_0)}/d)(x^2 + y^2)$. If $\left. \ell_1 \right|_{(z_0, \mu_0)} < 0$, then these periodic solutions are stable limit cycles, while if $\left. \ell_1 \right|_{(z_0, \mu_0)} > 0$, the periodic solutions are repelling.
\end{theorem}
The formulae for $\left. \ell_1 \right|_{(z_0, \mu_0)}$, the first Lyapunov coefficient, are given in Appendix \ref{app:LyapunovCoefficient}.
\begin{remark}
In the statement of Theorem \ref{thm:hopfThm} we used $\mu$ as the bifurcation parameter for consistency with the notation of \cite{JG-PJH:13}. In the analysis in this paper, $\epsilon_v$ plays the role of bifurcation parameter.
\end{remark}

The Hopf bifurcation theorem applies to our system, as summarized in Theorem \ref{thm:hopfResult} stated in Section \ref{sec:formalResults} and repeated below.

\textsc{Theorem} \ref{thm:hopfResult}. \textit{The system $\dot{z}_r = f_r(z_r, \epsilon_v)$ defined by  \eqref{eq:dynamicsDirectFeedback} has a deadlock equilibrium $z_{rd}$ given by \eqref{eq:reducedDeadlockEq}. For $\sigma > 6$, the dynamics undergo a Hopf bifurcation resulting in stable periodic solutions at $(z_{rd}, \epsilon_{v,0})$, where $\epsilon_{v,0}$ is the unique solution $\epsilon_v \in [0, 1/2]$ of $4(2-\sigma) \epsilon_v^2 -4(2-\sigma) \epsilon_v + (6-\sigma) = 0$. For $\sigma = 8$, $\epsilon_{v,0} = (3-\sqrt{6})/6 \approx 0.09175$.}

\begin{proof}[Proof of Theorem \ref{thm:hopfResult}]
Let $\epsilon_{v,0}$ be the solution $\epsilon_v \in [0, 1/2]$ of $4(2-\sigma) \epsilon_v^2 - 4(2-\sigma) \epsilon_v + (6-\sigma) = 0$, which exists for $\sigma > 6$. By Lemma \ref{lem:imaginaryEigenvalues}, the Jacobian $J_0$ of the system $\dot z_r = f(z_r, \epsilon_v)$ evaluated at the deadlock equilibrium $z_{rd}(\epsilon_v)$ has a simple pair of pure imaginary eigenvalues when $\epsilon_v = \epsilon_{v,0}$. Therefore, the first condition of the Hopf bifurcation theorem is satisfied.

Lemma \ref{lem:non-degeneracy} establishes that $\left. \dd (\mathrm{Re}\ \lambda(\epsilon_v))/\dd \epsilon_v \right|_{\epsilon_{v,0}} \neq 0$ for the two pure imaginary eigenvalues $\lambda$ so the second condition of the Hopf bifurcation theorem is satisfied. The result then follows: the system \eqref{eq:dynamicsDirectFeedback} undergoes a Hopf bifurcation as the parameter $\epsilon_v$ is lowered through its critical value $\epsilon_{v,0}$.

The first Lyapunov coefficient $\left. \ell_1 \right|_{(z_{rd}, \epsilon_{v,0})}$ is negative, as summarized by Lemma \ref{lem:lyapunovCoefficient}. This implies that the resulting limit cycles are stable.
\end{proof}

The following three lemmas, corresponding to the properties required by Theorem \ref{thm:hopfThm}, contain the detailed arguments behind the proof of Theorem \ref{thm:hopfResult}.

\begin{lemma} \label{lem:imaginaryEigenvalues}
Let $\sigma > 6$ and let $J_0  := D_{z_r} f_r ( z_{rd}, \epsilon_{v})$ be the Jacobian of the system $\dot z_r = f_r(z_r, \epsilon_v)$ defined by \eqref{eq:dynamicsDirectFeedback} evaluated at the deadlock equilibrium $z_{rd}$ given by \eqref{eq:reducedDeadlockEq}, considered as a function of $\epsilon_v$. Then, $J_0$ has a simple pair of two pure imaginary eigenvalues $\lambda(\epsilon_v)$ and $\bar{\lambda}(\epsilon_v)$ when $\epsilon_v = \epsilon_{v,0}$, where $\epsilon_{v,0}$ is the unique solution $\epsilon_v \in [0, 1/2]$ of $4(2-\sigma)\epsilon_v^2 - 4(2-\sigma)\epsilon_v + (6-\sigma) = 0$.
\end{lemma}

\begin{proof}
See Appendix \ref{app:imaginaryEigenvalues}.
\end{proof}

\begin{lemma} \label{lem:non-degeneracy}
Let $\lambda(\epsilon_v)$ and $\bar{\lambda}(\epsilon_v)$ be the simple pair of pure imaginary eigenvalues and let $\epsilon_{v,0}$ be as defined in Lemma \ref{lem:imaginaryEigenvalues}. Then, $\left. \dd (\mathrm{Re}\ \lambda(\epsilon_v)) /\dd \epsilon_v \right|_{\epsilon_{v,0}} < 0$.
\end{lemma}

\begin{proof}
See Appendix \ref{app:non-degeneracy}.
\end{proof}

\begin{lemma} \label{lem:lyapunovCoefficient}
Let $\ell_1 = \left. \ell_1 \right|_{(z_0, \epsilon_0)}$ be the first Lyapunov coefficient of the dynamics \eqref{eq:dynamicsDirectFeedback} evaluated at the deadlock equilibrium $z_{rd}$ given by \eqref{eq:reducedDeadlockEq}. Then $\left. \ell_1 \right|_{(z_0, \epsilon_0)} < 0.$
\end{lemma}

\begin{proof}
See Appendix \ref{app:HopfCriticality}.
\end{proof}

Theorem \ref{thm:hopfResult} then follows as a consequence of Lemmas \ref{lem:imaginaryEigenvalues}, \ref{lem:non-degeneracy}, and \ref{lem:lyapunovCoefficient}.

The implication of Theorem \ref{thm:hopfResult} is that the system \eqref{eq:dynamicsDirectFeedback} resulting from the limit $\epsilon_{\lambda} \to 0$ has a Hopf bifurcation at $\epsilon_v^*(0) = \epsilon_{v,0}$. However, as can be seen in Figure \ref{fig:smallLimitCycle}, limit cycle behavior persists for finite $\epsilon_{\lambda}$. One can numerically compute the eigenvalues of the linearization of the system \eqref{eq:mean-differenceDynamics} evaluated at the deadlock equilibrium $z_d$ and numerically show that a Hopf bifurcation occurs at a value $\epsilon_v^*(\epsilon_{\lambda})$. Figure \ref{fig:bifurcationValue} shows the numerically-computed values of $\epsilon_v^*(\epsilon_{\lambda})$ for a range of values of $\epsilon_{\lambda}$. It is clear that the numerical value for the limit $\lim_{\epsilon_{\lambda} \to 0} \epsilon_v^*(\epsilon_{\lambda})$ coincides with the analytical value $\epsilon_{v,0}$.

\section{Reduction to a planar limit cycle in the joint limit $\epsilon_{\lambda}, \epsilon_v \to 0$}\label{sec:reduction}
The results from Section \ref{sec:Hopf} strongly suggest the existence of a stable limit cycle for finite $\epsilon_v, \epsilon_{\lambda}$. In this and the following section we make this conclusion rigorous by performing a series of reductions collapsing the dynamics \eqref{eq:mean-differenceDynamics} to a planar system in the joint limit $\epsilon_{\lambda} \to 0, \epsilon_v \to 0$. A Poincar\'e-Bendixson argument affords the conclusion that the planar system exhibits a stable limit cycle. Then, in the next section, we show that this limit cycle persists for small but finite $\epsilon$ by applying results from geometric singular perturbation theory.

\subsection{A five dimensional attracting invariant submanifold}
In our first reduction, formalized in Lemma \ref{lem:asymptoticConvergence}, we note that $x_2$ must asymptotically converge to $0$ independent of the other states' behavior and that $x_1$ is attracted to the interval $[-1,1]$. Geometrically, this can be interpreted as the robot being attracted to the convex hull of its goal states $x_1^*, x_2^*$. This observation reveals an attracting invariant submanifold of dimension five whose restriction dynamics we then study.

We begin by considering the dynamics of $x_2$ independently of the other five dynamical variables, which gives us a nonautonomous system $\dot x_2 = f_{x_2}(t, x_2).$ The following results from \cite{JPL:68} concern the asymptotic behavior of a nonautonomous system
\beq \label{eq:nonautonomousSystem}
\dot x = f(t, x)
\eeq
defined on $G \subseteq \bbR^n$. Let $G^*$ be an open set of $\bbR^n$ containing $\bar G$, the closure of $G$. We assume that $\map{f}{[0, \infty) \times G^*}{\bbR^n}$ is a continuous (nonautonomous) vector field.

\begin{definition} \label{def:LyapunovFunction}
Let $\map{V}{[0, \infty) \times G^*}{\bbR}$ be a continuous, locally Lipschitz function. The function $V$ is said to be a \emph{Lyapunov function} of \eqref{eq:nonautonomousSystem} on $G$ if
\renewcommand{\theenumi}{\roman{enumi}}
\begin{enumerate}
\item given $x$ in $\bar G$ there is a neighborhood $N$ of $x$ such that $V(t,x)$ is bounded from below for all $t \geq 0$ and all $x$ in $N \cap G$.
\item $\dot V(t, x) \leq - W(x) \leq 0$ for all $ t \geq 0$ and all $x$ in $G$, where $W$ is continuous on $\bar G$. For $t$ where $V(t, x(t))$ is not differentiable, $\dot V$ is defined using the right-hand limit.
\end{enumerate}
If $V$ is a Lyapunov function for \eqref{eq:nonautonomousSystem} on $G$, we define
\[ E = \{ x ; W(x) = 0, x \in \bar G \} \text{ and } E_{\infty} = E \cup \{ \infty \}. \]
\end{definition}

The statement of Theorem \ref{thm:lasalle}, below, requires the notion of absolute continuity, which is defined as follows.
\begin{definition}[{\cite[Section 6.4]{HLR-PMF:10}}] \label{def:absoluteContinuity}
A real-valued function $f$ on a closed, bounded interval $[a,b]$ is said to be \emph{absolutely continuous} on $[a, b]$ provided for each $\epsilon > 0$, there is a $\delta > 0$ such that for every finite disjoint collection $\{ (a_k, b_k) \}_{k=1}^n$ of open intervals in $(a,b)$,
\[ \text{if } \sum_{k=1}^n [b_k - a_k] < \delta, \text{ then } \sum_{k=1}^n | f(b_k) - f(a_k) | < \epsilon. \]
\end{definition}
Lipschitz continuity implies absolute continuity, as follows:
\begin{proposition}[{\cite[Section 6.4, Proposition 7]{HLR-PMF:10}}] \label{prop:Lipschitz}
If the function $f$ is Lipschitz on a closed, bounded interval $[a,b]$, then it is absolutely continuous on $[a,b]$.
\end{proposition}

\begin{theorem}[{\cite[Theorem 1]{JPL:68}}] \label{thm:lasalle}
Let $V$ be a Lyapunov function for \eqref{eq:nonautonomousSystem} on $G$, and let $x(t)$ be a solution of \eqref{eq:nonautonomousSystem} that remains in $G$ for $t \geq t_0 \geq 0$ with $[t_0, \omega)$ the maximal future interval of definition of $x(t)$.
\renewcommand{\theenumi}{\alph{enumi}}
\begin{enumerate}
\item If for each $p \in \bar G$ there is a neighborhood $N$ of $p$ such that $| f(t,x) |$ is bounded for all $t \geq 0$ and all $x$ in $N \cap G$, then either $x(t) \to \infty$ as $t \to \omega^-$, or $\omega = \infty$ and $x(t) \to E_{\infty}$ as $t \to \infty$.
\item If $W(x(t))$ is absolutely continuous and its derivative is bounded from above (or from below) almost everywhere on $[t_0, \omega)$ and if $\omega = \infty$, then $x(t) \to E_{\infty}$ as $t \to \infty$.
\end{enumerate}
\end{theorem}

We now show that $x_2$ converges. For clarity of exposition, we write the argument as a series of lemmas.

\begin{lemma} \label{lem:Ginvariant}
Let $\varepsilon > 0$ and $\mfd = \domain \times \Delta^2 \times \bbR \times \bbR_+$, and let the set $G$ be the subset of $\mfd$ defined by
\beq \label{eq:defG} 
G := \{ z \in \mfd | \vbar \geq 1/2, \mbar > \varepsilon \}.
\eeq
The set $G$ is positive invariant under the dynamics $\dot z = f_z(z)$ defined by \eqref{eq:mean-differenceDynamics}.
\end{lemma}
\begin{proof}
Let $z = (x_1, x_2, \delm, \mbar, \delv, \vbar)$ be coordinates for $\mathcal{M} = \domain \times \Delta^2 \times \bbR \times \bbR_+$ and consider the dynamics $\dot z = f_z(z)$ defined by \eqref{eq:mean-differenceDynamics}.

Recall that $\varphi_1(x) = \| x - x_1^* \|^2/2, \varphi_2(x) = \| x - x_2^* \|^2/2$ with $x_i^* = [\pm 1, 0]^T \in \domain$. Direct calculation shows that $\pbar(x) = (\varphi_1(x) + \varphi_2(x))/2 = (1 + \| x \|^2)/2 \geq 1/2$.

Furthermore, recall from \eqref{eq:vBarDot} that $\dot \vbar = -\lambda (\vbar-\pbar)$, so $\dot \vbar(\vbar = 1/2) = -\lambda(1/2 - \pbar) \geq 0$ by the lower bound on $\pbar$. Therefore the set $\{ z \in \mfd | \vbar \geq 1/2 \}$ is positive invariant.

Similarly, note that $\mbar \geq 0$ by definition and that $-2\mbar \leq \delm \leq 2 \mbar$, so $\mbar = 0$ implies that $\delm = 0$. Therefore, from \eqref{eq:mBarDot}, $\dot \mbar (\mbar = 0) = \vbar/\epsilon_v$, so $\vbar \geq 1/2$ implies that $\dot \mbar(\mbar=0) \geq c/2\epsilon_v$ and therefore that $\mbar > 0$. Therefore, the continuity of the $\dot \mbar$ dynamics implies that for $\vbar > 1/2$, there exists an $\varepsilon > 0$ such that $\mbar < \varepsilon$ implies that $\dot \mbar(\mbar) > 0$. This implies that $G$ is a positive-invariant set.
\end{proof}

We now study the dynamics of $x_2$, which can be written as
\beq
\label{eq:x2Dott}
\dot{x}_2 = -2 \mbar(t) x_2.
\eeq

\begin{lemma} \label{lem:GLyapunovFn}
Let $G$ be the set defined by \eqref{eq:defG} in Lemma \ref{lem:Ginvariant}. Let $\varepsilon > 0$ and $W(z) = 2\varepsilon z_2^2 = 2 \varepsilon x_2^2$. The function $V(z) = \frac{1}{2}z_2^2 = \frac{1}{2} x_2^2$ is a Lyapunov function of \eqref{eq:x2Dott} on $G$.
\end{lemma}

\begin{proof}
Note that $V(z), W(z) \geq 0$, so they satisfy requirement i of Definition \ref{def:LyapunovFunction}. Computing $\dot V$ on $G$, we have 
\begin{align*}
\dot{V} &= \partial V/\partial x_2 \dot{x}_2 = -2 \mbar(t) x_2^2 \\
&\leq -2 \varepsilon x_2^2 = -W(z) \leq 0.
\end{align*}
The inequality derives from the definition of $G$ and shows that $V$ satisfies the second requirement of Definition \ref{def:LyapunovFunction}, thereby establishing the result.
\end{proof}

\begin{lemma} \label{lem:WabsCont}
The function $W(x_2(t))$ is absolutely continuous in $t$.
\end{lemma}
\begin{proof}
Proposition \ref{prop:Lipschitz} shows that Lipschitz continuity implies absolute continuity. It is well known that a function is Lipschitz continuous if it has bounded first derivative, so we proceed by bounding the derivative $\dd W(x_2(t))/\dd t$.

The time derivative is
\begin{align*} \frac{\mathrm{d}W}{\mathrm{d}t} = \frac{\partial W}{\partial x_2} \dot x_2 &= (4 \varepsilon x_2) (-2 \mbar(t) x_2) \\
&= -8\varepsilon \mbar(t) x_2^2.
\end{align*}
This implies that 
\beq \label{eq:magnitudeBound1}
|\dd W/\dd t| = 8 \varepsilon \mbar(t) x_2^2 = 16 \varepsilon \mbar(t) V(x_2(t)),
\eeq
where $V$ is the Lyapunov function defined in Lemma \ref{lem:GLyapunovFn}. Then $|\dd W/\dd t|$ is clearly lower bounded by zero and upper bounded by $8 \varepsilon V(x_2(0))$ since $V$ is a Lyapunov function for \eqref{eq:x2Dott}.
\end{proof}

Lemmas \ref{lem:Ginvariant}--\ref{lem:WabsCont} imply that $x_2 \to 0$. Furthermore, $x_1$ is attracted to the interval $[-1,1]$, as formalized in the following lemma.
\begin{lemma} \label{lem:x1Attracting}
Let $x_1(t) = z_1(t)$ be the first component of a solution of \eqref{eq:mean-differenceDynamics} with initial condition $z(0) \in G$, where $G$ is as defined in \eqref{eq:defG}. Then, $x_1(t) \to [-1, 1]$ as $t \to \infty$.
\end{lemma}

\begin{proof}
Let $\varepsilon > 0$ and note that $G$ is positive-invariant under \eqref{eq:mean-differenceDynamics}. The dynamics of $x_1$ are $\dot x_1 = \delm - 2 \mbar x_1$. From the definitions of $\delm$ and $\mbar$ as coordinates for the simplex $\Delta^2$ it is easily shown that $-2\mbar \leq \delm \leq 2 \mbar$. Then, on $G$, $\dot x_1$ can be lower bounded as
\begin{align*}
\dot x_1 &= \delm - 2\mbar x_1 \\
& \geq -2\mbar - 2\mbar x_1 \\
&= -2 \mbar (1+x_1) \\
& \geq -2 \varepsilon (1 + x_1)
\end{align*}
which is strictly positive for $x_1 < -1$. Similarly, $\dot x_1$ can be upper bounded such that $\dot x_1 > 0$ if $x_1 > 1$. Therefore, $x_1(t) \to [-1,1]$ as $t \to \infty$.
\end{proof}

Putting the preceding pieces together, we find that there is an attracting invariant submanifold of dimension five, as formalized in the following lemma.

\begin{lemma} \label{lem:asymptoticConvergence}
The variables $x_1$ and $x_2$ asymptotically converge to the interval $[-1, 1]$ and the point 0, respectively. Therefore, the five-dimensional submanifold $\mfd_5$ defined by $\mfd_5 = \{ z = (x_1, x_2, \delm, \mbar, \delv, \vbar) \in \mfd | x_1 \in [-1,1], x_2 = 0, \mbar > 0, \vbar \geq 1/2 \}$ is an attracting invariant submanifold under the (autonomous) dynamics \eqref{eq:mean-differenceDynamics}.
\end{lemma}

\begin{proof}
Lemmas \ref{lem:GLyapunovFn} and \ref{lem:WabsCont} show that $V$ is a Lyapunov function for the dynamical system $\dot x_2 = f_{x_2}(x_2)$ whose total derivative $\dot V$ is upper bounded by $-W(x_2)$. Therefore the conditions of Theorem \ref{thm:lasalle} hold.

Applying Theorem \ref{thm:lasalle}, we find that $x_2$ converges to the unique root of $W(x_2)$, which is located at $x_2 = 0$. Applying Lemma \ref{lem:x1Attracting}, we conclude that $x_1 \to [-1, 1]$. Therefore, $\mfd_5$ is an attracting subset of the set $G$ defined in \eqref{eq:defG}.
\end{proof}

\subsection{A two-dimensional slow manifold in the limit $\epsilon_{\lambda} \to 0, \epsilon_{v} \to 0$}
We now proceed to eliminate $\delv$ and $\mbar$ by taking the singular perturbation limit $\epsilon_{\lambda} \to 0, \epsilon_{v} \to 0$. The first limit couples $\delv$ to $\delp(x)$, as in Section \ref{sec:Hopf}, while the second limit pushes $\mbar$ to a slow manifold $\mbar(\delp(x), \delm)$. This results in a planar ``singular'' system.

We study the dynamics \eqref{eq:mean-differenceDynamics} restricted to $\mfd_5$ defined in Lemma  \ref{lem:asymptoticConvergence} by singular perturbation in the small parameter $\epsilon_v$. Further, let $\ell > 0$ be a positive number such that $\epsilon_{\lambda} = 1/\lambda = \ell \epsilon_v$. For a given value of $\ell$, this links $\epsilon_\lambda$ to $\epsilon_v$. Defining\footnote{This definition, which we make for consistency with the literature on singular perturbation problems, overloads the definition of $x \in \domain$ from Section \ref{sec:model}.} $x = (x_1, \delm) \in \mathcal{X} := [-1,1]^2$ and $y = ((1-2 \mbar)/\epsilon_v, \delv, \vbar) \in \bbR_+ \times \bbR \times \bbR_+$, the restricted dynamics can be written in the standard form for a singular perturbation problem:
\begin{align}
\dot{x} &= f_x(x,y,\epsilon_v) \label{eq:slowDynamics} \\
\epsilon_v \dot{y} &= g_y(x, y, \epsilon_v), \label{eq:fastDynamics}
\end{align}
where the slow dynamics is given by 
\begin{align*}
\dot{x}_1 = f_{x,1}(x, y, \epsilon_v) &= x_2 - (1-\epsilon_v y_1) x_1 \\
\dot{x}_2 = f_{x,2}(x, y, \epsilon_v) &= -\epsilon_v \left( \frac{1- \epsilon_v y_1 + x_2}{2y_3 + y_2} - \frac{1- \epsilon_v y_1 - x_2}{2 y_3 - y_2} \right) \\
& \ \ \ \  + x_2 y_1 y_3 + y_1 y_2 \left( \frac{3-\epsilon_v y_1}{2} \right),
\end{align*}
and the fast dynamics is given by
\begin{align*}
\epsilon_v \dot{y}_1 = g_{y,1}(x, y, \epsilon_v) =&\ \epsilon_v \left( \frac{1- \epsilon_v y_1 + x_2}{2y_3 + y_2} + \frac{1- \epsilon_v y_1 - x_2}{2 y_3 - y_2} \right) \\
& \ - \frac{y_1}{2} \left( (2 y_3 + y_2) \left( 1 + \frac{1 - \epsilon_v y_1 + x_2}{2} \right) \right. \\
& \left. \ \ \ \ \ \ \ \ + (2 y_3 - y_2) \left( 1 + \frac{1 - \epsilon_v y_1 - x_2}{2} \right) \right) \\
& +\ \frac{\sigma}{2} ((1 - \epsilon_v y_1)^2 - x_2^2) \\
\epsilon_v \dot{y}_2 = g_{y,2}(x, y, \epsilon_v) =& - \frac{1}{\ell}(y_2 - \delp) = -\frac{1}{\ell}(y_2 + 2 x_1)\\
\epsilon_v \dot y_3 = g_{y,3}(x, y, \epsilon_v) =& -\frac{1}{\ell}(y_3 - \pbar) = - \frac{1}{\ell} \left( y_3 - \frac{1 + x_1^2}{2} \right).
\end{align*}
It can be easily verified that the vector field $f_x$ points inward on the boundary of $\mathcal{X} = [-1, 1]^2$, so $\mathcal{X}$ is positive invariant.

The slow manifold $\mfd_s$ is given by the implicit function solution $y = h_y(x)$ of the limiting fast dynamics $g(x, h_y(x), 0) = 0$:
\begin{align}
y_1 &= h_{y,1}(x) := \frac{\sigma (1- x_2^2)}{3 (1+ x_1^2) - 2 x_1 x_2} \label{eq:y1Definition} \\
y_2 &= h_{y,2}(x) := -2 x_1 \nonumber \\
y_3 &= h_{y,3}(x) := \frac{1 + x_1^2}{2}.
\end{align}
Figure \ref{fig:slowManifoldError} compares the analytically-computed slow manifold $h_{y,1}(x(t))$ to the value of $y_1(t)$ computed based on a numerical simulation of the full six-dimensional system \eqref{eq:mean-differenceDynamics}. The small value of the error between the two values shows that the low-dimensional slow manifold is a good approximation to the trajectory of the high-dimensional system. Figure \ref{fig:slowManifoldSurface} shows the analytically-computed slow manifold surface $h_{y,1}(x)$ along with the trajectories $h_{y,1}(x(t))$ and $y_1(t)$.

\begin{figure}[htbp]
   \centering
   \includegraphics[width=3.5in]{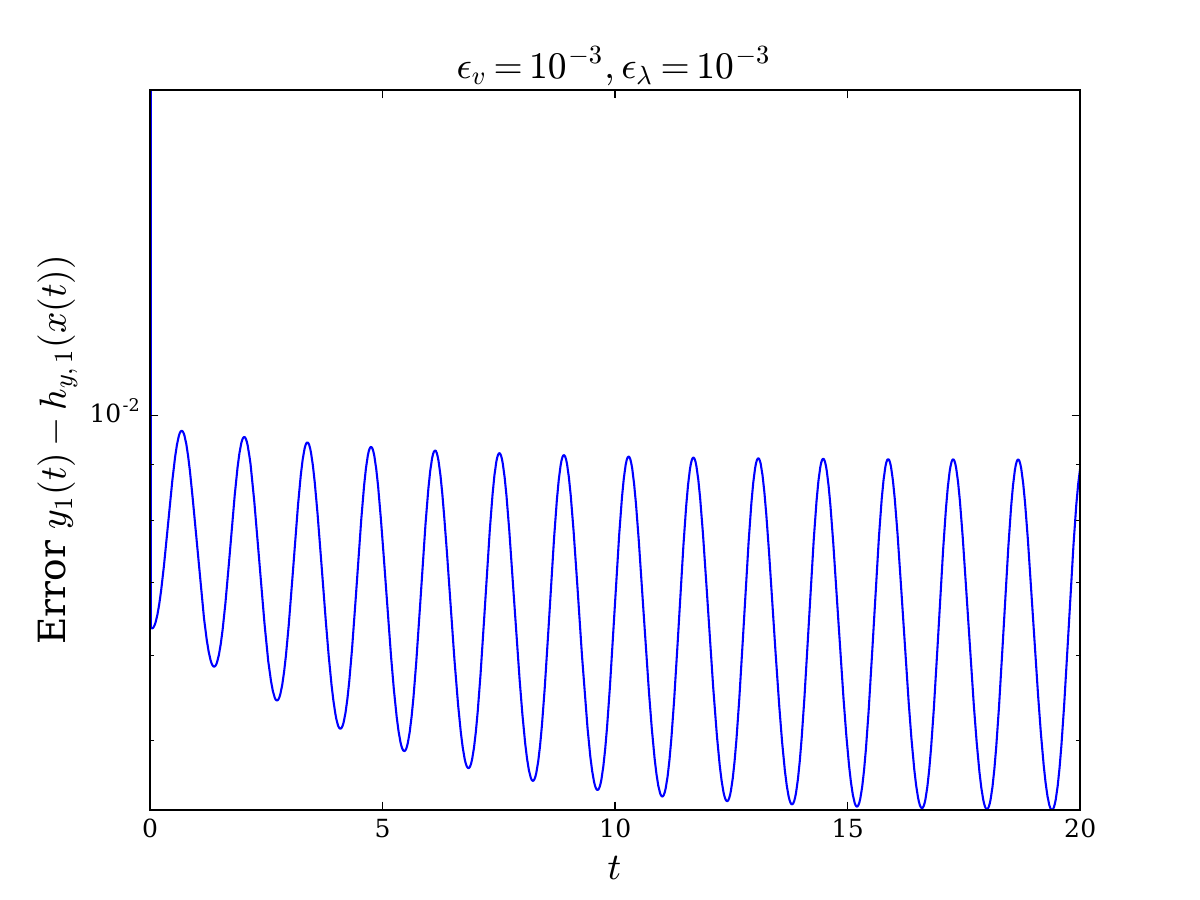} 
   \caption{Error in the slow manifold approximation \eqref{eq:y1Definition} computed for $\epsilon_v = 10^{-3}, \epsilon_{\lambda} = 10^{-3}$, plotted on a logarithmic scale. The blue trace shows the error $y_1(t) - h_{y,1}(x(t))$ for $y_1(t) = (1-2\mbar(t))/\epsilon_v$ computed from a simulated trajectory of the full six-dimensional system \eqref{eq:mean-differenceDynamics} and $h_{y,1}(x(t))$, the analytical expression for the slow manifold evaluated along the same trajectory. The small magnitude of the error shows that the analytical slow manifold is a good approximation to the full system.}
   \label{fig:slowManifoldError}
\end{figure}

\begin{figure}[htbp]
   \centering
   \includegraphics[width=3.5in]{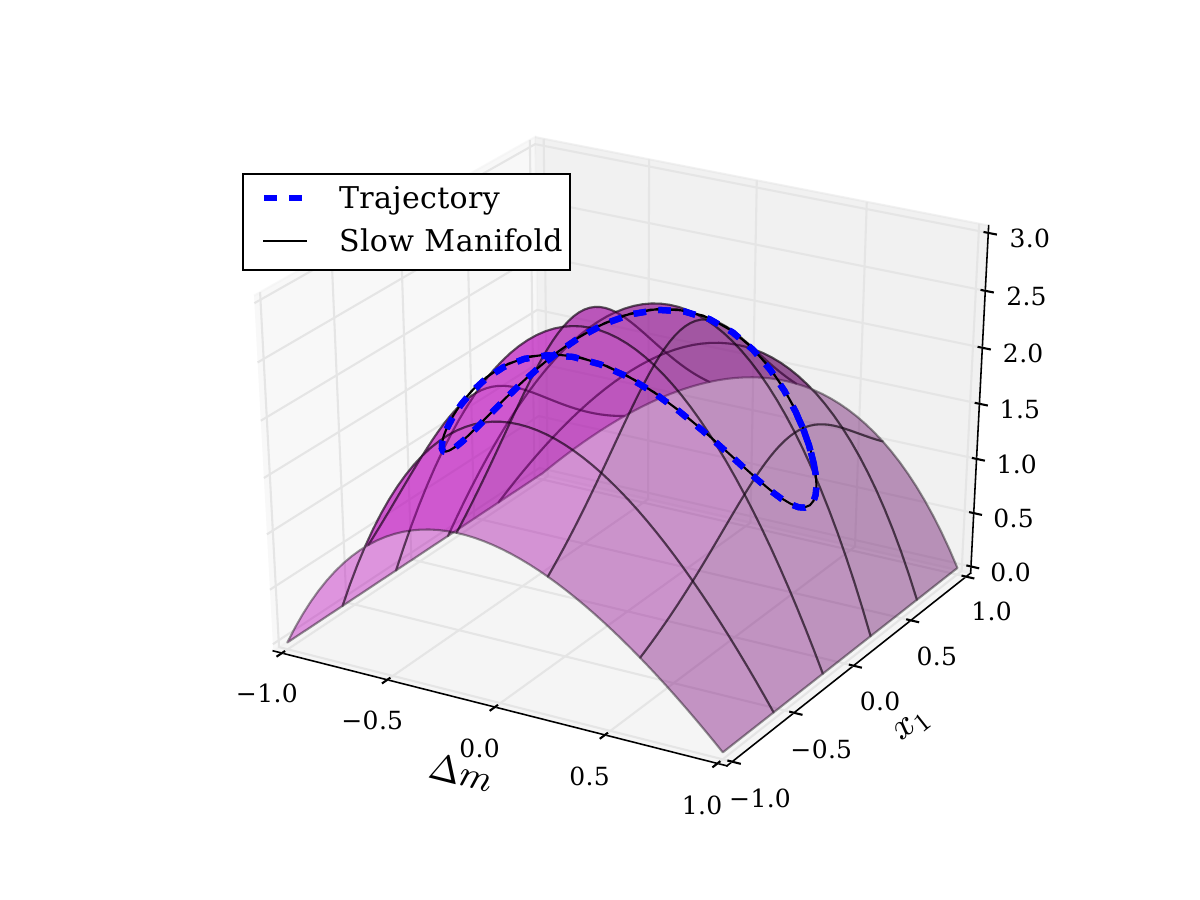} 
   \caption{Orbits on the slow manifold $\mathcal{M}_s$ defined by \eqref{eq:y1Definition} computed for $\epsilon_v = 10^{-3}, \epsilon_{\lambda} = 10^{-6}$. The blue trace shows $y_1(t) = (1-2\mbar(t))/\epsilon_v$ computed from a simulated trajectory of the full six-dimensional system \eqref{eq:mean-differenceDynamics}, while the magenta trace shows $y_1(t) = h_1(x(t))$, the analytical expression for the slow manifold evaluated along the same trajectory. The shaded surface shows the slow manifold surface $h_1(x)$. The close correspondence between the two traces shows that the analytical slow manifold is a good approximation to the full system.}
   \label{fig:slowManifoldSurface}
\end{figure}

The planar reduced dynamics on the slow manifold are given by the restriction of \eqref{eq:slowDynamics} to the slow manifold $\mfd_s$, now expressed in the coordinates of $\mathcal{X} = [-1,1]^2$,
\beq \label{eq:planarDynamics}
\dot x = f_x(x, h_y(x), 0),
\eeq
where the components are
\begin{align*}
\dot x_1 &= f_{x,1}(x, h_y(x), 0) = x_2 - x_1\\
\dot x_2 &= f_{x,2}(x, h_y(x), 0) = \frac{\sigma}{2}\frac{(1-x_2^2) (x_2 - 6 x_1 + x_1^2 x_2)}{3 x_1^2 - 2 x_1 +3}.
\end{align*}

As seen in Appendix \ref{app:fenichel}, the slow manifold $\mfd_s$ is hyperbolic if the initial layer equation, \eqref{eq:layerEquation}, $\dot{\delta y} = D_y g_y(x, 0, 0) \delta y$ (where $\delta y = y - h(x)$) has a hyperbolic equilibrium at the origin. Away from the boundary of the set $\mathcal{X} = [-1,1]^2$, the eigenvalues of the linearization $\left. \frac{\partial g_y}{\partial y} \right|_{(x,0,0)}$ are equal to $-1/\ell$ (with multiplicity two) and $-(3 - 2 x_1 x_2 + 3 x_1^2)/2$. By definition $\ell> 0$, so all three eigenvalues are strictly negative for $x$ in the interior of $\mathcal{X}$. Therefore they never intersect the imaginary axis and the slow manifold $\mfd_s$ is hyperbolic on the interior of $\mathcal{X}$. For points $x$ on the boundary of $\mathcal{X}$, the eigenvalues are not well defined, so we cannot conclude that $\mfd_s$ is hyperbolic at those points.

\subsection{Existence of a stable limit cycle in the planar system}
Now, we study the planar dynamics \eqref{eq:planarDynamics} and show by a straightforward application of the Poincar\'e-Bendixson Theorem \cite[Theorem 1.8.1]{JG-PJH:13} that they exhibit a  periodic orbit --- a limit cycle --- attracting an open annular neighborhood of the origin. We conjecture (and all numerical evidence corroborates) that this is an asymptotically stable limit cycle comprising the forward limit set of the entire origin-punctured state space. For present purposes it suffices to observe formally that an open neighborhood of initial conditions around the origin must take their forward limit set on this limit cycle. 

\begin{lemma}\label{lem:stableOrbit}
For $\sigma > 6,$ there exists a periodic orbit of the reduced system \eqref{eq:planarDynamics} comprising the forward limit set of an open annular neighborhood excluding the (unstable) origin.
\end{lemma}
\begin{proof}
Note that the set $\mathcal{X} = [-1, 1]^2$ is invariant set under the reduced dynamics $\dot x = f_x(x, h_y(x), 0)$. Furthermore, note that the only equilibria in $\mathcal{X}$ are the origin and the two corners $(x_1, x_2) = (1, 1), (-1, -1)$. It is easy to see that the interior of $\mathcal{X}$ is an invariant set; the edges $x_2 = 0$ and $x_2 = 1$ are part of the stable manifolds of the equilibria at the corners and the vector field is directed inwards on the other two edges. Let $\mathring{\mathcal{X}}$ be the interior of $\mathcal{X}$. Then $\mathring{\mathcal{X}}$ is pre-compact, connected, and contains a single fixed point at the origin.

The linearization of the reduced dynamics at the origin is given by
\[ J_{r,0} = \begin{bmatrix} -1 & 1 \\ -\sigma & \sigma/6 \end{bmatrix}. \]
The determinant $\det J_{r,0} = 5\sigma/6 > 0$ for all $\sigma > 0$ and the trace $\tr J_{r,0} = \sigma/6 - 1 > 0$ for all $\sigma > 6$. The trace and determinant are, respectively, the sum and product of the two eigenvalues, so the fact that they are both positive implies that the eigenvalues are themselves positive. Therefore, for $\sigma > 6$, the origin is an unstable focus.

Since the entire annular region, $\mathring{\mathcal{X}} \setminus \{0 \}$, is a pre-compact, positive invariant set  possessing no fixed points it follows from the Poincar\'e-Bendixson Theorem that its forward limit set consists of proper (non-zero period) periodic orbits. In particular, in the neighborhood of the excluded repeller at the origin, there must exist an isolated periodic orbit comprising the forward limit set for that entire (punctured) neighborhood.\footnote{Although some texts, e.g., \cite[Ch. 7.0]{SS:94}, reserve the term ``limit cycle'' for isolated periodic orbits, the preponderant usage seems to favor that introduced in, e.g., \cite[Ch. 10.6]{MH-SS-RD:04}, \cite[Ch. VII.1]{PH:64}, \cite[Ch. V.9]{CR:99}, requiring merely that a periodic orbit comprise the limit set for some disjoint set of initial conditions, as we show here. Conjecture \ref{conj:Floquet}, below, will advance numerical evidence to support our assumption hereafter that this periodic orbit is hyperbolic, in which case it must be isolated \cite[Ch. I.4.4]{CCC:78}, and hence qualifies unreservedly for the usage ``limit cycle.''}
\end{proof}

\section{Persistence of the limit cycle for finite $\epsilon_v, \epsilon_{\lambda}$}\label{sec:persistence}
We now give conditions (which were previously stated as Theorem \ref{thm:mainResult} in Section \ref{sec:formalResults}) under which the limit cycle whose existence was proven in the limit $\epsilon_{\lambda} \to 0, \epsilon_v \to 0$ in Lemma \ref{lem:stableOrbit} persists for finite values of $\epsilon_{\lambda}, \epsilon_v$. The result depends upon the conjectured hyperbolicity of that cycle, for which we establish numerical evidence below.

\vspace{1mm}
\textsc{Theorem} \ref{thm:mainResult}. \textit{
Accepting Conjecture \ref{conj:Floquet}, below, for $\sigma>6$, there exists a stable limit cycle of \eqref{eq:mean-differenceDynamics} for sufficiently small, but finite, values of $\epsilon_{\lambda}$ and $\epsilon_v$. Equivalently, fixing $\lambda$, there exists a stable limit cycle of \eqref{eq:mean-differenceDynamics} for sufficiently large, but finite, values of $v^*$.}

Lemma \ref{lem:stableOrbit} shows that the singularly perturbed system with $\epsilon \to 0$, i.e., $\epsilon_v, \epsilon_{\lambda} \to 0$ exhibits a limit cycle $\gamma_0$. The following result due to Fenichel \cite{NF:79} then allows us to show that this limit cycle persists for sufficiently small $\epsilon_{v}, \epsilon_{\lambda} > 0$. 

Two pieces of notation are required to state the result. The symbol $\eqMfd_H$ represents the open set on which the linearization of the dynamics normal to the slow manifold has hyperbolic fixed points. In $\eqMfd_H$ the reduced vector field $X_R$ is defined by
\[
X_R(m) = \pi^{\eqMfd} \partial/\partial \epsilon X^{\epsilon}(m)|_{\epsilon=0},
\]
where $\pi^{\eqMfd}$ is the projection onto $\eqMfd$ defined in Equation \eqref{eq:reducedVectorField} of Appendix \ref{app:fenichel}. We can now state the result.

\begin{theorem}[{\cite[Theorem 13.1]{NF:79}}] \label{thm:Fenichel}
Let $M$ be a $C^{r+1}$ manifold, $2 \leq r \leq \infty$. Let $X^{\epsilon}, \epsilon \in (-\epsilon_0, \epsilon_0)$ be a $C^r$ family of vector fields, and let $\eqMfd$ be a $C^r$ submanifold of $M$ consisting entirely of equilibrium points of $X^0$. Let $\gamma \in \eqMfd_H$ be a periodic orbit of the reduced vector field $X_R$, and suppose that $\gamma_0$, as a periodic orbit of $X_R$, has 1 as a Floquet multiplier of multiplicity precisely one. Then there exists $\epsilon_1 > 0$ and there exists a $C^{r-1}$ family of closed curves $\gamma_{\epsilon}, \epsilon \in (-\epsilon_1, \epsilon_1)$, such that $\gamma_0 = \gamma$ and $\gamma_{\epsilon}$ is a periodic orbit of $\epsilon^{-1} X^{\epsilon}$. The period of $\gamma_{\epsilon}$ is a $C^{r-1}$ function of $\epsilon$.
\end{theorem}

Theorem 13.2 of \cite{NF:79} states that, when $\gamma_0$ is hyperbolic, the stability of the family $\gamma_{\epsilon}$ of periodic orbits for $\epsilon > 0$ can be deduced from the stability of $\gamma_0$ and the stability of the linearization of $f_x, g_y$ at $\epsilon = 0$.

\begin{conjecture} \label{conj:Floquet}
Let $\sigma = 8$. The periodic orbit, $\gamma_0$, whose existence is guaranteed by Lemma \ref{lem:stableOrbit} for the reduced dynamics $\dot x = f_x(x, h_y(x), 0)$ defined by \eqref{eq:planarDynamics} is asymptotically stable with Floquet multipliers $\rho_1 = 1, \rho_2 < 1$.
\end{conjecture}

The Floquet multiplier associated with perturbations along the vector field is always equal to 1, and the remaining $n-1$ multipliers are the eigenvalues of the linearized Poincar\'e map $DP$ of the periodic orbit \cite[Section 1.5]{JG-PJH:13}. We proceed by computing the numerical approximation to $P$ and $DP$ for the case $\sigma=8$ using the Poincar\'e section $\Sigma = \{ (x_1, 0) | 0.05 < x_1 < 1 \}$. The results are shown in Figure \ref{fig:poincareMap}. From studying numerical solutions of the reduced dynamics, we know that the periodic orbit crosses $\Sigma$ at a point $p \approx 0.34934$, at which point the linearized map $DP(p)$ has a value of approximately 0.48. This implies that the two Floquet multipliers are $\rho_1 = 1, \rho_2 \approx 0.48 < 1$.

\begin{figure}[htbp]
   \centering
   \includegraphics[width=3.5in]{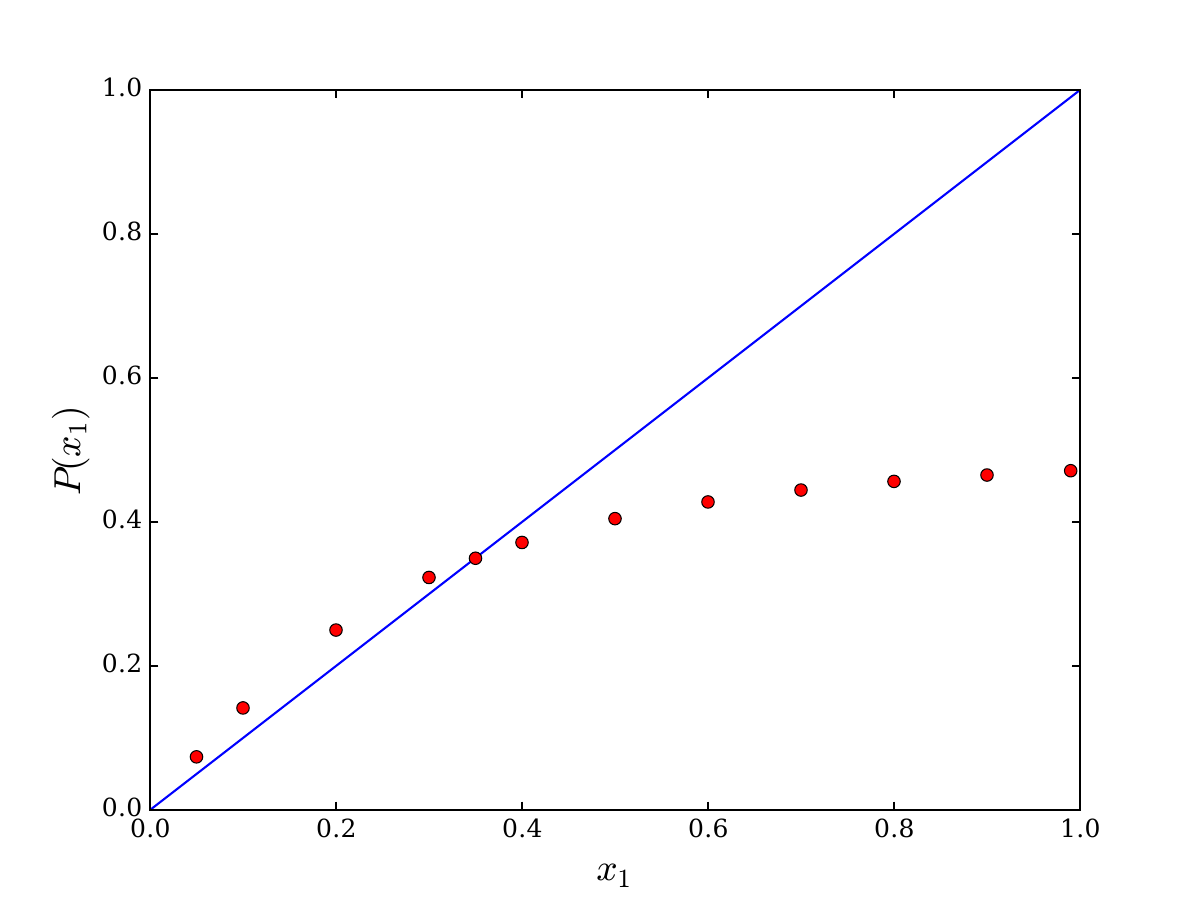} 
   \caption{Poincar\'e map of \eqref{eq:planarDynamics} computed for the section $\Sigma = \{ (x_1, 0) | 0.05 < x_1 < 0.9 \}$; the blue line represents the identity map. At the fixed point $x_{1,c} \approx 0.34934$, the slope $DP(x_{1,c}) \approx 0.48$, which is consistent with the limit cycle being stable.}
   \label{fig:poincareMap}
\end{figure}

In Conjecture \ref{conj:Floquet}, we only make claims about the Floquet multipliers in the case where $\sigma$ takes its nominal value $\sigma=8$. However, numerical investigation suggests that this result holds for generic $\sigma > 6$. Theorem \ref{thm:mainResult} then follows from applying Theorem \ref{thm:Fenichel} to the (now conjectured to be asymptotically stable) limit cycle found in Lemma \ref{lem:stableOrbit}.
\begin{proof}[Proof of Theorem {\ref{thm:mainResult}}]
Let $\gamma$ be a periodic orbit of the reduced system whose existence is shown in Lemma \ref{lem:stableOrbit}. By Conjecture \ref{conj:Floquet}, $\gamma$ has 1 as a Floquet multiplier of multiplicity precisely 1. On $\mathring{\mathcal{X}}$ the eigenvalues $\mu_1, \mu_2, \mu_3$ of the linearization $\partial g_y/\partial y$ can be computed in closed form and take the values
\[ \mu_1 = \mu_2 = -\frac{1}{\ell}, \ \mu_3 = -\frac{(3 -2 x_1 x_2 + 3 x_1^2)}{2}. \]
On the interior of the slow manifold $(x_1, x_2) \in \mathring{\mathcal{X}} = (-1,1)^2$ these are bounded away from the imaginary axis, so $\gamma \in \eqMfd_H$. Then, by Theorem \ref{thm:Fenichel}, there exists $\epsilon_1 > 0$ and a family of periodic orbits $\gamma_{\epsilon}, \epsilon \in (-\epsilon_1, \epsilon_1)$ such that $\gamma_0 = \gamma$.

Specifically, for each $\epsilon \in (0, \epsilon_1)$, there exists a stable limit cycle $\gamma_{\epsilon}$ with $\epsilon_v = \epsilon$ and $\epsilon_{\lambda} = \ell \epsilon$. Equivalently, fix $\lambda>0$ and define $v_1^* = 1/\epsilon_1 < +\infty$. Then for $v^* > v_1^*$, there exists a stable limit cycle $\gamma_{\epsilon}$ for $\epsilon= \epsilon_v = 1/v^*$.
\end{proof}

\begin{corollary} \label{cor:lambdaAxis}
Theorem \ref{thm:mainResult} establishes the existence of a stable periodic orbit $\gamma$ of the reduced dynamics for $(\epsilon_\lambda, \epsilon_v) \in \bbR_+ \times [0, \epsilon_1)$, i.e., the neighborhood of the $\epsilon_\lambda$ axis for sufficiently small $\epsilon_v > 0$ and generic $\epsilon_\lambda$.
\end{corollary}
\begin{proof}
Note that the fast dynamics \eqref{eq:fastDynamics} defines $\epsilon_\lambda = \ell \epsilon_v$, specifying only that $\ell > 0$. Therefore, the result of Theorem \ref{thm:mainResult} applies for parameter values $(\epsilon_\lambda, \epsilon_v) \in \{ (\ell \epsilon_v, \epsilon_v) | \epsilon_v \in [0, \epsilon_1), \ell \in \bbR_+ \}$.
\end{proof}

\section{Conclusion}\label{sec:conclusion}
In summary, we have developed a dynamical systems method for managing motivations in physically-embodied agents, i.e., robots. This method provides a novel way for a system to autonomously and continuously switch between a set of vector fields, each of which defines a possible dynamics for the physical state corresponding to its performing a navigation task.

We specialize to the case where the system has two vector fields defined on a simply-connected subset of $\bbR^2$. By imposing several symmetries on our system, we are able to analyze the system in the limit where first one, and then both, of two parameters approaches zero. In the joint limit we reduce the system to a planar dynamical system by means of a singular perturbation analysis; a Poincar\'e-Bendixson argument shows that this planar system exhibits an isolated periodic orbit corresponding to the physical system state oscillating between two fixed points, one for each of the two vector fields. By appealing to geometric singular perturbation theory, we show that this periodic orbit persists for finite values of the two parameters.

A natural extension of this work is to consider cases where the system has more than two navigation tasks and where the domain $\domain$ is punctured by obstacles, i.e., not simply connected. One natural way to extend this work to the case of multiple tasks is to decompose tasks into hierarchies encoded in binary trees; then, a variant of the system studied in this paper can run in each node to manage the motivations represented by each of its child nodes. By designing an appropriate method to feed the information from the child nodes back to their parent, it will be possible to maintain the limit cycle behavior for the larger number of tasks. Extending the analysis in this paper to the case of non-simply connected domains may prove more complex, as the analysis relies on several coordinate transformations that will be difficult to extend the more general case.

The other natural extension of this work is to apply it by implementing the motivational system on a physical robot. This implementation work is already in progress and will be the subject of a future report.

\appendix

\section{First Lyapunov coefficient calculation} \label{app:LyapunovCoefficient}
Kuznetsov \cite[Section 5.4]{YAK:13} provides the following formulae for computing $\left. \ell_1 \right|_{(z_0, \epsilon_0)}$, the first Lyapunov coefficient of the dynamics $\dot z = f(z, \epsilon)$. Let $J_0 = \left. D_z f \right|_{(z_0, \epsilon_0)}$. Property 1 of Theorem \ref{thm:hopfThm} implies that $J_0$ has two pure imaginary eigenvalues $\lambda(\epsilon_0), \bar{\lambda}(\epsilon_0) = \pm i \omega_0, \omega_0 > 0$. Let $q \in \bbC^n$ be a complex eigenvector corresponding to $\lambda(\epsilon_0)$:
\[ J_0 q = i \omega_0 q, \ J_0 \bar{q} = -i \omega_0 \bar{q}. \]
Introduce the adjoint eigenvector $p \in \bbC^n$ satisfying
\[ J_0^T p = -i \omega_0 p, \ J_0^T \bar{p} = i \omega_0 \bar{p} \]
and satisfying the normalization condition $\langle p, q \rangle = 1$, where $\langle \cdot, \cdot \rangle$ is the standard inner product on $\bbC^n$.

Then, Taylor expand $f(z) = f(z,\epsilon_0)$ to third order in $z$:
\[ f(z) = \frac{1}{2} B(z,z) + \frac{1}{6} C(z,z,z) + \bigO{\| z \|^4}, \]
where $B$ and $C$ are multilinear functions given by
\[ B_i(x,y) = \sum_{j,k=1}^n \left. \frac{\partial^2 f_i(\xi, \epsilon)}{\partial \xi_j \partial \xi_k} \right|_{\xi = 0} x_j y_k := \sum_{j,k=1}^n B_{ijk} x_j y_k, \]

\begin{align*}
C_i(x,y,z) &= \sum_{j,k,l=1}^n \left. \frac{\partial^3 f_i(\xi, \epsilon)}{\partial \xi_j \partial \xi_k \partial \xi_l} \right|_{\xi = 0} x_j y_k z_l \\
&:= \sum_{j,k,l=1}^n C_{ijkl} x_j y_k z_l,
\end{align*}
which define the coefficients $B_{ijk}$ and $C_{ijkl}$. The coefficient $\left. \ell_1 \right|_{(z_0, \epsilon_0)}$ is then given by Equation (5.62) of \cite{YAK:13}:
\begin{align} \label{eq:defLyapunovCoefficient}
\left. \ell_1 \right|_{(z_0, \epsilon_0)} = \frac{1}{2\omega_0} \mathrm{Re} &\left[ \langle p, C(q, q, \bar q) \rangle - 2 \langle p, B(q, J_0^{-1} B(q, \bar q)) \rangle \right. \\
&\left. + \langle p, B(\bar q , (2 i \omega_0 I-J_0)^{-1} B(q,q) ) \rangle \right], \nonumber
\end{align}
where $I$ is the identity matrix.

\section{Analysis of the Hopf bifurcation of \eqref{eq:dynamicsDirectFeedback}}
In this section, we report computations relevant to the results in Theorem \ref{thm:hopfResult}.

\subsection{Jacobian computation} \label{app:jacobian}
The following claim, which can be verified by direct computation, establishes the value of the Jacobian $J_0$ of the dynamics \eqref{eq:dynamicsDirectFeedback} evaluated at the deadlock equilibrium $z_{rd}$ given by \eqref{eq:reducedDeadlockEq}.

\begin{claim} \label{claim:jacobian}
Let $J_0 = \left. D_{z_r} f_r(z_r, \epsilon_v) \right|_{z_r = z_{rd}}$ be the Jacobian of \eqref{eq:dynamicsDirectFeedback} evaluated at the deadlock equilibrium defined by \eqref{eq:reducedDeadlockEq}. Then
\begin{align*}
J_0 &= \begin{bmatrix}
-2\mbar_{rd} & 0 & 1 & 0 \\
0 & -2\mbar_{rd} & 0 & 0 \\
j_{31} & 0 & j_{33} & 0 \\
0 & 0 & 0 & j_{44}
\end{bmatrix}
\end{align*}
where the non-zero components are given by
\begin{align}
j_{31} &= \left. \frac{\partial \dot \delm_r}{\partial x_{1r}} \right|_{z_r = z_{rd}} = -8 \epsilon_v \mbar_{rd} - 2(1-2 \mbar_{rd})(1+\mbar_{rd})/\epsilon_v \label{eq:j31}\\
j_{33} &= \left. \frac{\partial \dot \delm_r}{\partial \delm_r} \right|_{z_r = z_{rd}} = -2 \epsilon_v + (1-2\mbar_{rd})/(2\epsilon_v) \label{eq:j33}\\
j_{44} &= \left. \frac{\partial \dot \mbar_r}{\partial \mbar_r} \right|_{z_r = z_{rd}} =  -2(\epsilon_v + \sigma \mbar_{rd}) - (1+4\mbar_{rd})/(2\epsilon_v). \label{eq:j44} 
\end{align}

The characteristic polynomial $p_4(\lambda)$ of $J_0$ is given by the determinant $| J_0 - \lambda I |,$ where $I$ is the identity matrix. This determinant can be computed directly using expansion by minors:
\beq \label{eq:charPoly}
p_4(\lambda) = | J_0 - \lambda I | = (\lambda - j_{22}) (\lambda - j_{44}) (\lambda^2 + (2 \mbar_{rd} - j_{33}) \lambda - (j_{31} + 2\mbar_{rd} j_{33}).
\eeq
\end{claim}

\subsection{Proof of Lemma \ref{lem:imaginaryEigenvalues}} \label{app:imaginaryEigenvalues}

\begin{proof}[Proof of Lemma \ref{lem:imaginaryEigenvalues}]
The characteristic polynomial of $J_0$ is computed in \eqref{eq:charPoly} in Appendix \ref{app:jacobian} and can be expressed as 
\[ p_4(\lambda) = | J_0 - \lambda I | = (\lambda + 2\mbar_{rd}) (\lambda - j_{44}) p_2(\lambda),
\]
where the final factor is given by coefficients arising directly from specific entries of the (sparse) Jacobian as
\[ p_2(\lambda) = \lambda^2 + (2 \mbar_{rd}-j_{33}) \lambda - (j_{31} + 2 \mbar_{rd} j_{33}) \]
and the components $j_{kl}$ of the Jacobian are as given in Equations \eqref{eq:j31}--\eqref{eq:j44} of Appendix \ref{app:jacobian}.

The roots of $p_4$ are given by 
\[ \{ j_{22} = -2 \mbar_{rd}, j_{44} = -(1 + 4\epsilon_v^2)/(2 \epsilon_v) - 2 \mbar_{rd}/\epsilon_v -2 \sigma \mbar_{rd} \} \cup \{ \lambda | p_2(\lambda) = 0 \}.\]
It is clear that the first two roots are negative for all $\epsilon_v > 0$, so the stability properties of the deadlock equilibrium are determined by the roots of $p_2$.

The roots $\lambda_{\pm}$ of $p_2$ are purely imaginary if $2 \mbar_{rd} - j_{33} = 0$. The condition $2 \mbar_{rd} - j_{33}=0$ implies 
\beq \label{eq:bifnPoint}
\mbar_{rd}(\epsilon_v) = \frac{1}{2} - \epsilon_v,
\eeq
which is a valid value of $\mbar$ for $\epsilon_v \in [0, 1/2]$.
Inserting this expression into the expression \eqref{eq:mBarDeadlockEqn} for $\mbar_d(\epsilon_v)$, one finds that $2\mbar_{rd} - j_{33}=0$ implies that 
\beq \label{eq:bifnPointEpsilon}
4(2-\sigma)\epsilon_v^2 - 4(2-\sigma) \epsilon_v + (6-\sigma) = 0.
\eeq
For $\sigma > 6$, this equation has two real-valued solutions, of which only one is in the relevant interval $[0, 1/2]$ for which $\mbar_{rd}(\epsilon_v)$ is defined. Therefore this solution is the relevant one defining the bifurcation value.
\end{proof}

\subsection{Proof of Lemma \ref{lem:non-degeneracy}} \label{app:non-degeneracy}

\begin{proof}[Proof of Lemma \ref{lem:non-degeneracy}]
Let $\epsilon_{v,0}$ and $p_2(\lambda) = \lambda^2 + (2 \mbar_{rd}-j_{33}) \lambda - (j_{31} + 2 \mbar_{rd} j_{33})$ be defined as in the proof of Lemma \ref{lem:imaginaryEigenvalues}, where $j_{kl}$ are defined in Appendix \ref{app:jacobian}. Let $\Delta_p = (2\mbar_{rd}-j_{33})^2 + 4 (j_{31} + 2 \mbar_{rd} j_{33})$ be the discriminant of $p_2$ with respect to $\lambda$. 

At the bifurcation value $\epsilon_{v,0}$, $2 \mbar_{rd} - j_{33}=0$, so the discriminant $\Delta_p = 4(j_{31} + 4 \mbar_{rd}^2)$, which is negative, as can be shown by substituting in the expressions for $j_{kl}$ and grouping terms. This implies that the real part of the roots $\lambda_\pm$ are given by $(2\mbar_{rd}-j_{33})/2$. Therefore, since the coefficients of $p_2$ are continuous functions of $\epsilon_v$, the sign of the derivative $\dd(\mathrm{Re}\  \lambda(\epsilon_v))/ \dd \epsilon_v |_{\epsilon_v = \epsilon_{v,0}}$ is determined by that of  $(2\mbar'_{rd}-j'_{33})(\epsilon_{v,0}) := \dd (2 \mbar_{rd}-j_{33})/\dd\epsilon_v |_{\epsilon_v = \epsilon_{v,0}}.$ Computing the derivative, we get
\begin{align*} \left. (2\mbar'_{rd}-j'_{33}) \right|_{\epsilon_v = \epsilon_{v,0}} &= \left. \frac{ 2(1-2\epsilon_v)(1+2\epsilon_v)(\sigma-2) }{ 3 + 2(\sigma - 2) \epsilon_v + 4(1-\sigma) \epsilon_v^2 }  \right|_{\epsilon_v = \epsilon_{v,0}},
\end{align*}
which is non-zero (and in fact, positive) for all $\epsilon_v \in [0, 1/2]$.

Thus, the derivative $\dd(\mathrm{Re}\  \lambda(\epsilon_v))/ \dd \epsilon_v |_{\epsilon_v = \epsilon_{v,0}} < 0$.
\end{proof}

\subsection{Criticality of the Hopf bifurcation in Theorem \ref{thm:hopfResult}} \label{app:HopfCriticality}
The following result concerns the the Hopf bifurcation whose existence is proven in Theorem \ref{thm:hopfResult}. It implies that the Hopf bifurcation is supercritical, so the periodic solutions created by the bifurcation are stable limit cycles.

The remainder of this section constitutes the proof of Lemma \ref{lem:lyapunovCoefficient}. As in Appendix \ref{app:LyapunovCoefficient}, write $\left. \ell_1 \right|_{(z_0, \epsilon_0)} = \frac{1}{2\omega_0} \mathrm{Re}[ T_1 + T_2 + T_3 ],$ where, from \eqref{eq:defLyapunovCoefficient},
\begin{align*}
T_1 &= \langle p, C(q, q, \bar{q}) \rangle\\
T_2 &= - 2 \langle p, B(q, J_0^{-1} B(q, \bar q)) \rangle \\
T_3 &= \langle p, B(\bar q , (2 i \omega_0 I-J_0)^{-1} B(q,q) ) \rangle,
\end{align*}
$I$ is the identity matrix, and $\omega_0, p, q, B,$ and $C$ are as defined in Appendix \ref{app:LyapunovCoefficient}. We show that $\realPart{T_i} < 0$ for $\sigma > 6$. This implies that $\ell_1 < 0$, since $\omega_0 > 0$ by definition.

Let $J_0 = D_z \left. f \right|_{(z_0,\epsilon_0)}$ be the Jacobian of the dynamics \eqref{eq:dynamicsDirectFeedback} evaluated at the bifurcation point. As shown in \eqref{eq:bifnPoint}, $J_0$ has two purely imaginary eigenvalues when $2 \mbar_{rd} - j_{33} = (2(1+2\epsilon_v) \mbar_{rd} - (1-\epsilon_v^2))/(2\epsilon_v) = 0$. This implies that $1-2\mbar_{rd} = 2 \epsilon_v$ at the bifurcation point.

As shown in Claim \ref{claim:jacobian}, the Jacobian $J_0$ can be computed in closed form and takes the value
\begin{align} J_0 &= \begin{bmatrix}
-2\mbar_{rd} & 0 & 1 & 0 \\
0 & -2 \mbar_{rd} & 0 & 0 \\
j_{31} & 0 & j_{33} & 0 \\
0 & 0 & 0 & j_{44}
\end{bmatrix} \nonumber \\
&=
\begin{bmatrix}
-2\mbar_{rd} & 0 & 1 & 0 \\
0 & -2\mbar_{rd} & 0 & 0 \\
-6 + 8 \epsilon_v^2 & 0 & 1-2\epsilon_v & 0 \\
0 & 0 & 0 & (4\epsilon_v - 3)/(2 \epsilon_v) - \sigma (1-2\epsilon_v) - 2 \epsilon_v
\end{bmatrix}, \label{eq:jacobian}
\end{align}
where we have used the relationship established in \eqref{eq:bifnPoint} $1-2\mbar_{rd} = 2 \epsilon_v$ that holds at the bifurcation point in the final expression.

The eigenvalue problems $J_0 q = i \omega_0 q, \ J_0^T p = -i \omega_0 p$ can be solved analytically, yielding
\begin{align*}
\omega_0 &= \sqrt{ | j_{31} + 4 \mbar_{rd}^2 |} = \sqrt{5 + 4 \epsilon_v -12 \epsilon_v^2}, \\
q &= \begin{bmatrix} q_1 & 0 & q_3 & 0 \end{bmatrix}^T = \begin{bmatrix} \frac{-i}{2 \omega_0} & 0 & \frac{\omega_0 - 2 \mbar_{rd} i}{2\omega_0} & 0 \end{bmatrix}^T,\\
p &= \begin{bmatrix} p_1 & 0 & p_3 & 0 \end{bmatrix}^T = \begin{bmatrix} -(2\mbar_{rd} + \omega_0 i) & 0 & 1 & 0 \end{bmatrix}^T.
\end{align*}
The eigenvectors satisfy the required normalization condition $\langle p, q \rangle = 1$. The sparsity of $p$ and $q$ greatly simplifies the computations of $T_1, T_2, T_3$.

\subsubsection{Computing $T_1$}
First, we compute $T_1$. Recall from \eqref{eq:defLyapunovCoefficient} that $T_1 = \langle p, C(q, q, \bar{q}) \rangle$. Direct computation shows that $C_i(q, q, \bar q) = 0$ for $i \in \{1, 2, 4 \}$. The relevant components of $C_{3jkl}$ are $C_{3111} = -96 \epsilon_v \mbar_{rd}, C_{3113} = C_{3131} = C_{3311} = -12 \epsilon_v + (1-2\mbar_{rd})/\epsilon_v$. Substituting in the values of $p$ and $q$ then yields
\[ \mathrm{Re}[T_1] = -\frac{2(1-2\epsilon_v)(1-6\epsilon_v)}{4 \omega_0^2}. \]

\subsubsection{Computing $T_2$}
Next, we compute $T_2$. Recall from \eqref{eq:defLyapunovCoefficient} that $T_2 = - 2 \langle p, B(q, J_0^{-1} B(q, \bar q)) \rangle$. Direct computation shows that $B_i(q, \bar q) = 0$ for $i \in \{1, 2, 3\}$ and that $B_4(q, \bar q)$ has relevant terms $B_{411} = (1-2\mbar_{rd})(1+\mbar_{rd})/\epsilon_v -12 \epsilon_v \mbar_{rd}, B_{413} = B_{431} = -2 \epsilon_v - (1-2\mbar_{rd})/(2\epsilon_v), B_{433} = \sigma/2$. Then $B_4(q, \bar q) = ((3\sigma-9)+22 \epsilon_v - 4\sigma \epsilon_v^2)/(4\omega_0^2)$. The matrix $J_0^{-1}$ can be computed from \eqref{eq:jacobian} in closed form, and is equal to
\[ J_0^{-1} = \begin{bmatrix}
-2\mbar_{rd}/\omega_0^2 & 0 & 1/\omega_0^2 & 0\\
0 & -1/2\mbar_{rd} & 0 & 0 \\
j_{31}/\omega_0^2 & 0 & 2\mbar/\omega_0^2 & 0 \\
0 & 0 & 0 & 1/j_{44}
\end{bmatrix}, \]
where $j_{kl}$ are defined in \eqref{eq:jacobian}, so the only non-zero component of $J_0^{-1} B(q, \bar q)$ is the fourth one, which is equal to $B_4(q,\bar q)/j_{44} = ((3\sigma-9)+22 \epsilon_v - 4\sigma \epsilon_v^2)/(4\omega_0^2 j_{44})$. Direct computation then shows that $B_i(q, J_0^{-1} B(q, \bar q)) = 0$ for $i \in \{2, 4\}$ and then that
\begin{align*} \realPart{T_2} &= \realPart{-2 \langle p, B(q, J_1^{-1} B(q, \bar q)) \rangle} \\
&= \frac{1-2\epsilon_v}{4\omega_0^2} \frac{10+\sigma+\sigma \omega_0^2 -4(4+\sigma) \epsilon_v + 4(2+\sigma)\epsilon_v^2}{-3-2(\sigma-2)\epsilon_v + 4(\sigma-1)\epsilon_v^2}.
\end{align*}

\subsubsection{Computing $T_3$}
Finally, we compute $T_3$. Recall from \eqref{eq:defLyapunovCoefficient} that $T_3 = \langle p, B(\bar q , (2 i \omega_0 I-J_0)^{-1} B(q,q) ) \rangle$. As in the case of $T_2$, direct computation shows that $B_i(q, q) = 0$ for $i \in \{1, 2, 3\}$ and that $B_4(q, q)$ has relevant terms $B_{411}, B_{413} = B_{431}$, and $B_{433}$.

Let $\Gamma = (2i \omega_0 I - J_1)^{-1}$, which has the structure
\[ \Gamma = \begin{bmatrix}
\gamma_{11} & 0 & \gamma_{13} & 0 \\
0 & \gamma_{22} & 0 & 0 \\
\gamma_{31} & 0 & \gamma_{33} & 0 \\
0 & \gamma_{42} & 0 & \gamma_{44}
\end{bmatrix}, \]
where $\gamma_{44} = 1/(-j_{44} + 2 i \omega_0)$.
The first three components $(\Gamma B(q, q))_i = 0$ for $i \in \{1, 2, 3\}$ and the only non-zero component is the fourth one given by $\gamma_{44} B_4(q,q)$. Then direct computation shows that $B_i(\bar{q}, \Gamma B(q,q)) = 0$ for $i \in \{2, 4\}$ so that $T_3 = \bar p_1 B_1(\bar q, \Gamma B(q,q)) + B_3(\bar q, \Gamma B(q,q)).$ Straightforward but tedious calculations then show that 
\[ \mathrm{Re}[T_3] = -\frac{9(1+\sigma) + 2(5-7\sigma)\epsilon_v - 4(11-\sigma) \epsilon_v^2 -8(7-\sigma) \epsilon_v^3}{8\epsilon_v \omega_0^2}.\]

\subsection{Combining terms}
Combining the results for $T_1, T_2,$ and $T_3$, we get that 
\[ \mathrm{Re}[T_1 + T_2 + T_3] = \frac{p_5(\epsilon_v)}{8\epsilon_v \omega_0^2 (-3 -2(\sigma-2)\epsilon_v + 4(\sigma-1)\epsilon_v^2)}, \]
where $p_5(\epsilon_v) = 27(1+\sigma) +2 (13-24\sigma + 9 \sigma^2) \epsilon_v -8(40-9\sigma + 8 \sigma^2) \epsilon_v^2 + 16(26-17\sigma+4\sigma^2)\epsilon_v^2 + 304(\sigma-1) \epsilon_v^4 - 32(\sigma-1)^2 \epsilon_v^5$. Let $d(\epsilon_v) = 3 + 2(\sigma-2) \epsilon_v - 4(\sigma-1) \epsilon_v^2)$. Then
\begin{align*} 
\left. \ell_1 \right|_{(x_0, \mu_0)} &= \frac{1}{2\omega_0} \mathrm{Re}(T_1 + T_2 + T_3)\\
& = \frac{-1}{16 \epsilon_v \omega_0^3} \frac{p_5(\epsilon_v)}{d(\epsilon_v)},
\end{align*}
so the sign of $\ell_1$ is determined by the sign of the signs of $p_5(\epsilon_v)$ and $d(\epsilon_v)$ since $\epsilon_v, \omega_0 > 0$. We proceed by showing that both polynomials are positive. 

First, evaluating $d(\epsilon_v)$ at $\epsilon_v = \epsilon_v(\sigma)$ defined by \eqref{eq:bifnPointEpsilon}, we find that $d(\epsilon_v(\sigma)) = (2\sigma(\sqrt{\sigma - 2}-1)/(\sigma-2))$, which is clearly positive for $\sigma > 3$. Similarly, $p_5(\epsilon_v(\sigma=6)) = 189> 0$ and it can be shown that $\dd p_5(\epsilon_v(\sigma))/ \dd \sigma > 0$ for $\sigma > 6$, so $p_5(\epsilon_v(\sigma)) > 0$ for $\sigma > 6$. Therefore, $\ell_1 |_{(x_0, \mu_0)} < 0$ for $\sigma > 6$, which establishes the desired result.

\section{A tutorial on geometric singular perturbation theory} \label{app:fenichel}
Singular perturbation theory is a tool for studying dynamical systems characterized by two time scales, slow time $t$ and fast time $\tau$. The time scales are related by $\tau = t/\epsilon$, where $\epsilon > 0$ is a small parameter. In the slow time scale, the dynamical system is governed by differential equations that are singular at $\epsilon=0$. By taking the limit $\epsilon \to 0$, i.e., assuming that the fast dynamics are much faster than the slow dynamics, one can often reduce a system to the slow dynamics.

Fenichel did fundamental work on this theory, for which \cite{NF:79} is a relatively comprehensive reference. Of particular interest to this paper is the theory he developed that allows one to relate the behavior of (the invariant manifolds of) a system in the limit $\epsilon \to 0$ to the behavior with finite $\epsilon > 0$. In order to do this globally on a compact subset of the state space, Fenichel developed a geometric, or coordinate-free, notion of singular perturbation. The remainder of the section constitutes a summary of the relevant material in \cite{NF:79}. We begin by summarizing the local results, which are expressed in a given set of coordinates, before introducing the more abstract global, coordinate-free results.

\subsection{Local results}
Let $M$ be an open subset of $\bbR^{\mu} \times \bbR^{\nu}$, and let $\eqMfd = M \cap (\bbR^{\mu} \times \{ 0 \})$ be nonempty. We consider a system of the form
\begin{align}
\dot x &= f_0(x, y, \epsilon) \label{eq:standardSystem} \\
\epsilon \dot y &= g(x, y, \epsilon) \nonumber
\end{align}
where $\dot \ $ denotes differentiation with respect to $t$, defined for $(x,y) \in M$, for small, real $\epsilon$. When $\epsilon = 0$ the system \eqref{eq:standardSystem} degenerates to the reduced system
\begin{align}
\dot x &= f_0(x, y, 0) \label{eq:reducedSystem} \\
0 &= g(x, y, 0). \nonumber
\end{align}
The second equation of \eqref{eq:reducedSystem} is an implicit function that defines $y$ as a function of $x$. The relation can be expressed explicitly, at least locally, as a function $y = h(x)$ \cite{CJ:95}. The set $\{ (x, y) | y = h(x) \}$ is called the \emph{slow manifold}. By translating the $y$ coordinates by $- h(x)$, we can set $y=0$ on the slow manifold, which we denote by \eqMfd. Therefore, we assume that 
\beq \label{eq:equilibriumCondition}
g(x, 0, 0) = 0 \ \ \text{for all } (x, 0) \in \eqMfd,
\eeq
so that \eqref{eq:reducedSystem} defines a flow in \eqMfd, and we assume that this flow has a periodic orbit $\gamma_0: x = p(t), y = 0.$ Fenichel's aim is to describe the orbit structure of \eqref{eq:standardSystem} for small nonzero $\epsilon$.

By rescaling time to $\tau = t/\epsilon$, we can transform \eqref{eq:standardSystem} to
\begin{align}
x^{\prime} &= \epsilon f_0(x, y, \epsilon) \label{eq:fastSystem} \\
y^{\prime} &= g(x, y, \epsilon), \nonumber
\end{align}
where $\ ^{\prime}$ denotes differentiation with respect to $\tau$. The set $\eqMfd$ consists entirely of equilibrium points of the system \eqref{eq:fastSystem} in the limit $\epsilon \to 0$.

The plan is to relate the orbit structure of \eqref{eq:standardSystem} near $\gamma_0$, for small nonzero $\epsilon$, to the orbit structure of the reduced system \eqref{eq:reducedSystem} near $\gamma_0$ and to the linearization of $\lim_{\epsilon \to 0}$ \eqref{eq:fastSystem} at points of $\gamma_0$. The linearization of $\lim_{\epsilon \to 0}$ \eqref{eq:fastSystem} at $(x, 0) \in \eqMfd$ is
\beq \label{eq:linearizedFastSystem}
\begin{bmatrix}
\delta x \\
\delta y
\end{bmatrix}^{\prime} = 
\begin{bmatrix}
0 & 0 \\
0 & D_2 g(x, 0, 0)
\end{bmatrix}
\begin{bmatrix}
\delta x \\
\delta y
\end{bmatrix},
\eeq
where $D_2 g(x, 0, 0)$ denotes differentiation with respect to the second argument of $g$ evaluated at $(x, 0, 0)$. The second component satisfies
\beq \label{eq:layerEquation}
\delta y^{\prime} = D_2 g(x, 0, 0) \delta y,
\eeq
a linear equation parametrized by $(x, 0) \in \eqMfd$. We refer to \eqref{eq:layerEquation} (Equation (3.8) of \cite{NF:79}) as the \emph{initial layer equation.}

The first qualitative question that Fenichel asks about \eqref{eq:standardSystem} is whether it has a periodic orbit $\gamma_{\epsilon}$ near $\gamma_0$ for $\epsilon$ near zero. Fenichel \cite{NF:79} claimed that Anosov \cite{DVA:63} obtained the most general result in the literature. In particular \cite[Section III]{NF:79}, Anosov proved that $\gamma_0$ can be continued to a family $\gamma_{\epsilon}$ if: (i) $\gamma_0$, regarded as a periodic orbit of the reduced system \eqref{eq:reducedSystem}, has 1 as a Floquet multiplier of multiplicity precisely one, and (ii) for each $(x, 0) \in \gamma_0$, the initial layer equation \eqref{eq:layerEquation} has a hyperbolic equilibrium point at $\delta y = 0$. The first condition can be interpreted as a nondegeneracy requirement on the periodic orbit $\gamma_0$ itself, while the second condition is sometimes called normal hyperbolicity of the slow manifold defined by $g(x, y, 0) = 0$. Theorem 13.1 of \cite{NF:79} makes this result precise.

\subsection{Global results}
The definitions up to here have been in a given set of coordinates. In order to properly account for limit cycles, Fenichel develops a global, coordinate-free notion of the singular perturbation problem. Let $M$ be a $C^{r+1}$ manifold, $1 \leq r \leq \infty$. Let $X^{\epsilon}: M \to TM$ be a family of vector fields on $M$, parametrized by $\epsilon \in (-\epsilon_0, \epsilon_0)$, such that $X^{\epsilon}$ is a $C^r$ function of $(m, \epsilon)$. Let $\eqMfd$ be a $C^r$ submanifold of $M$ consisting entirely of equilibrium points of $X^0$, and let $z = \phi(m)$ be a $C^{r+1}$ local coordinate in $M$. In $z$-coordinates the flow of $X^{\epsilon}$ satisfies
\beq \label{eq:slowFlowGlobalCoords}
z^{\prime} = X^{\epsilon} \phi(\phi^{-1}(z))
\eeq
subject to the condition
\[ X^{\epsilon} \phi(\phi^{-1}(z)) = 0 \ \ \text{for } z \in \phi(\eqMfd). \]

Let $\mu$ be the dimension of $\eqMfd$ and let $\nu$ be the codimension of $\eqMfd$ in $M$. Because $X^0$ vanishes identically on $\eqMfd$, $T_m\eqMfd$ is in the kernel of $T X^0(m)$ for any $m \in M$. In coordinates, then, $T x^0(m)$ will have $\mu$ eigenvalues equal to zero and $\nu$ additional eigenvalues, which we call the \emph{nontrivial eigenvalues}. The subspace $T_m \eqMfd$ is invariant under $TX^0(m)$, and so $TX^0(m)$ induces a linear map
\[ Q X^0(m): T_m M/ T_m \eqMfd \to T_m M / T_m \eqMfd \]
on the quotient space. The eigenvalues of $Q X^0(m)$ are the nontrivial eigenvalues of the linearization of $\lim_{\epsilon \to 0}$ \eqref{eq:slowFlowGlobalCoords} at $z = \phi(m)$.

Let $\eqMfd_R$ be the open set where $QX^0$ is invertible. For each $m \in \eqMfd_R$, $T_m \eqMfd$ has a unique complement $N_m$ which is invariant under $T X^0(m)$. Let $\pi^{\eqMfd}$ denote the projection on $T \eqMfd$ defined by the splitting $TM | \eqMfd_R = T \eqMfd \oplus N.$  Let $\eqMfd_H \subset \eqMfd_R$ be the open subset where $QX^0$ has no pure imaginary eigenvalues; this is the normally-hyperbolic component of the slow manifold.

In $\eqMfd_R$ the reduced vector field $X_R$ is defined by
\beq \label{eq:reducedVectorField}
X_R(m) = \pi^{\eqMfd} \partial/\partial \epsilon X^{\epsilon}(m)|_{\epsilon=0}.
\eeq

Now we can state the main theorem that asserts conditions under which periodic orbits of the reduced vector field $X_R$ defined in the limit $\epsilon \to 0$ persist for $\epsilon > 0$.
\begin{theorem}[{\cite[Theorem 13.1]{NF:79}}]
Let $M$ be a $C^{r+1}$ manifold, $2 \leq r \leq \infty$. Let $X^{\epsilon}, \epsilon \in (-\epsilon_0, \epsilon_0)$ be a $C^r$ family of vector fields, and let $\eqMfd$ be a $C^r$ submanifold of $M$ consisting entirely of equilibrium points of $X^0$. Let $\gamma \in \eqMfd_H$ be a periodic orbit of the reduced vector field $X_R$, and suppose that $\gamma_0$, as a periodic orbit of $X_R$, has 1 as a Floquet multiplier of multiplicity precisely one. Then there exists $\epsilon_1 > 0$ and there exists a $C^{r-1}$ family of closed curves $\gamma_{\epsilon}, \epsilon \in (-\epsilon_1, \epsilon_1)$, such that $\gamma_0 = \gamma$ and $\gamma_{\epsilon}$ is a periodic orbit of $\epsilon^{-1} X^{\epsilon}$. The period of $\gamma_{\epsilon}$ is a $C^{r-1}$ function of $\epsilon$.
\end{theorem}

For many applications, we are only interested in the case of small positive $\epsilon$. In \cite[Section V]{NF:79}, Fenichel explains how he is able to obtain results for $\epsilon \in (-\epsilon_0, \epsilon_0)$. Furthermore, the stability results of \cite[Theorem 13.2]{NF:79} are stated for $\epsilon > 0$. Let us now discuss how the theorem is applied. The main conditions are 1) that the periodic orbit $\gamma_0$ be contained in $\eqMfd_H$, the normally-hyperbolic component of the slow manifold, and 2) that $\gamma_0$ have 1 as a Floquet multiplier of multiplicity precisely one.

If one has a global coordinate system for $\eqMfd$, testing for normal hyperbolicity reduces to verifying that the eigenvalues of $\partial g/\partial y |_{\eqMfd}$ have non-zero real parts; if the real parts are negative, $\eqMfd$ is stable. The Floquet multipliers of $\gamma_0$ are the eigenvalues of $B$, the linearized Poincar\'e map of $\gamma_0$, so a multiplier of 1 corresponds to a fixed point of the Poincar\'e map, and multipliers less than (greater than) 1 correspond to stability (instability) of the orbit. There are $\mu$ Floquet multipliers $\rho_i, i \in \oneto{\mu}$, where $\mu$ is the dimension of $\eqMfd$. It can be shown \cite{JG-PJH:13} that $\det(B) = \prod_{i=1}^{\mu} \rho_i = \exp \int_0^T \tr (A(s)) \mathrm{d}s$, where $T$ is the period of the periodic orbit and $A(s)$ is the linearization of the reduced dynamics $Df(x, 0, 0)|_{\gamma_0(s)}$. The existence of the periodic orbit means that there is one Floquet multiplier equal to 1. In general, $\rho_i$ have to be found by numerically computing $B$, unless one can bound the sign of $\tr A$ on the slow manifold. Alternatively, if one can show that the limit cycle is asymptotically stable on the slow manifold, the Floquet multiplier condition follows.

\section*{Acknowledgement}
This work was supported in part by Air Force Research Laboratory grant FA865015D1845 (subcontract 669737-1).

\bibliographystyle{abbrv}

\end{document}